\documentclass[11pt]{article}
\usepackage{amsmath,amsthm,amssymb}
\usepackage{enumerate}
\usepackage{hyperref}
\usepackage{tikz}
\usepackage[mathlines]{lineno}
\usepackage{dsfont} % needed for double-struck 1

\numberwithin{figure}{section}   
\numberwithin{table}{section}   
\numberwithin{equation}{section}   

\theoremstyle{plain}
\newtheorem{thm}{Theorem}[section]
\newtheorem{lem}[thm]{Lemma}
\newtheorem{prop}[thm]{Proposition}

\newtheorem{obs}[thm]{Observation}

\newtheorem{conj}[thm]{Conjecture}
\newtheorem{quest}[thm]{Question}

\theoremstyle{definition}
\newtheorem{ex}[thm]{Example}
\newtheorem{rem}[thm]{Remark}

\definecolor{red}{rgb}{.8,0,0}
\def\red{\color{red}}
\definecolor{bblu}{rgb}{0,0,1}

\def\mtx#1{\begin{bmatrix} #1 \end{bmatrix}}

\newcommand{\R}{{\mathbb R}}
\newcommand{\C}{{\mathbb C}}

\newcommand{\Rnn}{\R^{n\times n}}
\newcommand{\re}{\operatorname{Re}}

\newcommand{\tr}{\operatorname{tr}}
\newcommand{\diam}{\operatorname{diam}}
\newcommand{\W}{\operatorname{W}} %%%%%%%New

\newcommand{\dig}{\Gamma} %%%%%NEW
\newcommand{\cp}{\, \Box\,} %%%%%NEW

\newcommand\lexp{\,{\textcircled{\scriptsize{L}}} \,}

\usepackage[bb=boondox]{mathalfa} % for double struck 1 and 0
\newcommand{\bone}{\mathbb{1}} % double-struck 1
\newcommand{\bv}{{\bf v}}
\newcommand{\bw}{{\bf w}}
\newcommand{\bz}{{\bf z}}
\newcommand{\bx}{{\bf x}}
\newcommand{\by}{{\bf y}}
\newcommand{\bd}{{\bf d}}
\newcommand{\bzero}{{\bf 0}}

\newcommand{\dsone}{\ensuremath{\mathds{1}}} % double-struck 1
\newcommand{\A}{\mathcal A}
\newcommand{\nL}{\mathcal L}
\newcommand{\D}{\mathcal D}
\newcommand{\DL}{\D^L}
\newcommand{\DQ}{{\D^Q}}
\newcommand{\nDL}{\D^{\nL}}
\newcommand{\lam}{\lambda}
\newcommand{\mult}{\operatorname{mult}}

\newcommand{\proj}{\operatorname{proj}}
\newcommand{\spec}{\operatorname{spec}}
\newcommand{\dspec}{\operatorname{spec}_{\D}}
\newcommand{\dLspec}{\operatorname{spec}_{\DL}}
\newcommand{\dQspec}{\operatorname{spec}_{\DQ}}
\newcommand{\ndLspec}{\operatorname{spec}_{\nDL}}
\newcommand{\dev}{\partial}

\newcommand{\dlev}{\partial^L}
\newcommand{\dqev}{\partial^Q}
\newcommand{\ndlev}{\partial^{\nL}}

\newcommand{\sgn}{\operatorname{sgn}}

\newcommand{\ol}{\overline}
\newcommand{\ds}{\displaystyle}
\newcommand{\ms}{\medskip}
\newcommand{\bit}{\begin{itemize}}
\newcommand{\eit}{\end{itemize}}
\newcommand{\ben}{\begin{enumerate}}
\newcommand{\een}{\end{enumerate}}
\newcommand{\beq}{\begin{equation}}
\newcommand{\eeq}{\end{equation}}
\newcommand{\bea}{\begin{eqnarray*}}
\newcommand{\eea}{\end{eqnarray*}}
\newcommand{\bpf}{\begin{proof}}
\newcommand{\epf}{\end{proof}}
\newcommand{\x}{\times}
\newcommand{\lc}{\left\lceil}
\newcommand{\rc}{\right\rceil}

\newcommand{\lp}{\!\left(}
\newcommand{\rp}{\right)}
\newcommand{\lsb}{\left\{}
\newcommand{\rsb}{\right\}}

\begin{document}
% \linenumbers

\title{Spectra of variants of distance matrices of graphs and digraphs: a survey}
\author{Leslie Hogben\thanks{Department of Mathematics, Iowa State University,
Ames, IA 50011, USA and American Institute of Mathematics, 600 E. Brokaw Road, San Jose, CA 95112, USA
(hogben@aimath.org).}\and Carolyn Reinhart\thanks{Department of Mathematics, Iowa State University,
Ames, IA 50011, USA
(reinh196@iastate.edu).}\and} 

\maketitle

\begin{abstract} 
 Distance matrices of graphs were introduced by Graham and Pollack in 1971 to study a problem in communications. Since then, there has been extensive research on the distance matrices of graphs -- a 2014 survey by Aouchiche and Hansen on spectra of distance matrices of graphs lists more than 150 references. In the last ten years, variants such as the distance Laplacian, the distance signless Laplacian, and the normalized distance Laplacian matrix of a graph have been studied. After a brief description of the early history of the distance matrix and its motivating problem, this survey focuses on comparing and contrasting techniques and results for the four types of distance matrices.  Digraphs are treated separately after the discussion of graphs, including discussion of similarities and differences between graphs and digraphs.  New results are presented that complement existing results, including results for some the  matrices on unimodality of characteristic polynomials for graphs, preservation of parameters by cospectrality for graphs, and bounds on spectral radii for digraphs. \end{abstract}



\section{Introduction}\label{sintro}

The study of spectral graph theory from a mathematical perspective began in the 1950s and separately in quantum chemistry in 1931 \cite{CRS97}. The first papers considered the adjacency matrix (which provides a natural description of a graph as a matrix) and the study of its spectrum (multiset of eigenvalues); formal definitions of this and other matrices associated with a graph are given below.  Various applications led to the study of additional matrices associated with the graph, including the Laplacian, signless Laplacian, and normalized Laplacian. Entire books have appeared on spectral graph theory, including \cite{CRS97}, \cite{BH}, \cite{B14}, and \cite{N18}.

Distance matrices were introduced by Graham and Pollack in \cite{GP71} to study the problem of routing messages through  circuits; this problem is discussed in more detail in Section \ref{ss:D-loop}.  
There has been a lot of research on distance matrices themselves (Aouchiche and Hansen's survey of results through 2014 in \cite{AH14} is 85 pages and contains more than 150 references).  The concept of distances in a graph has been used in applications much longer.  For example, the Wiener index, a graph parameter readily computed from the distance matrix, was introduced in 1947 in chemical graph theory by Wiener in \cite{W47} (where it is called the path number and used to determine boiling points).

More recently,  several   variants of the distance matrix that parallel the variants of the adjacency matrix have been defined and studied: Aouchiche and Hansen introduced the distance signless Laplacian and distance Laplacian in \cite{AH14} and Reinhart introduced the normalized distance Laplacian in \cite{R20}.  The focus of this survey  is to compare and contrast results and techniques for four matrices: the distance matrix, the distance signless Laplacian, the  distance Laplacian, and the normalized distance Laplacian of a graph. Additionally, some new results are presented to fill in gaps in the literature. With the exception of the discussion of the historical background and motivation in Section \ref{ss:D-loop} (where only the distance matrix is discussed), we address all four matrices and organize this article by topic; e.g., results concerning cospectrality are presented in Section \ref{s:cospec} for these four matrices.  Since a graph must be connected for all  distances to be finite, we assume all graphs discussed are connected. Digraphs are handled separately in Section \ref{s:dig}.

We begin in Section \ref{s:D-loop-GL} by describing Graham and Pollak's motivation for studying the distance matrix and early results, including their addressing scheme for loop switching and the unimodality conjecture. We prove that the sequence of coefficients of the distance signless Laplacian and normalized distance Laplacian are log-concave and the sequence of  their absolute value is unimodal. In Section \ref{s:method}, techniques for computing spectra are discussed, including the use of twin vertices, the relationship of the distance matrix to the adjacency matrix for graphs with low diameter, matrix products, and other linear algebraic techniques. We prove a new result regarding how twin vertices can be used to determine the spectra of a graph in Section \ref{ss:quot}.

Well-studied classes of graphs, including strongly regular, distance regular, and transmission regular graphs, are discussed in Section \ref{s:SRG}.  In Section \ref{s:fam}, we provide the spectra of several well-known families of graphs. We apply our result about twin vertices to determine the spectrum of a star graph with an added edge for the normalized distance Laplacian. The spectral radii of the distance matrix and its variants are discussed in Section \ref{s:rho}, including the extremal values and graphs which achieve these values. We establish that the spectral radii of $\D$, $\DQ$, and $\DL$ are edge monotonically strictly decreasing (they was previously known to be edge monotonically decreasing). Furthermore, we prove bounds on the distance matrix in terms of the transmission and show the bounds are tight if and only if the graph is transmission regular.

In Section \ref{s:cospec}, we discuss known results regarding the number of graphs with a cospectral mate, graphs determined by their spectra, parameters preserved by cospectrality, and cospectral constructions. We provide new examples that show several parameters are not preserved by various distance matrices and we show  that transmission regular graphs that are distance cospectral must have the same transmission and Wiener index. Finally, in Section \ref{s:dig}, we provide an overview of results for digraphs, many of which mirror known results for graphs. We also establish bounds on the spectral radius of the distance Laplacian and normalized distance Laplacian of digraphs.

Next we provide precise definitions of terms used throughout.  Let $G$ be a graph with vertex set $V(G)=\{v_1,\dots,v_n\}$ and edge set $E(G)$ (an {\em edge} is a set of two distinct vertices); the edge $\{v_i,v_j\}$ is often denoted by $v_iv_j$.  The number of vertices is the {\em order} of $G$. For a graph $G$ (but not a digraph), all the matrices associated with $G$ that we discuss are real and symmetric and so all the eigenvalues are real. The {\em adjacency matrix} of $G$, denoted by  $\A(G)=[a_{ij}]$, is the $n \times n$ matrix with $a_{ij}=1$ if  $v_i v_j\in E(G)$, and $a_{ij}=0$ if $v_i v_j\not\in E(G)$.  %Since $\A(G)$ is nonnegative, its {\em spectral radius} $\rho$  (i.e., eigenvalue of largest magnitude) is positive.  Furthermore, $\rho$ is simple because $\A(G)$ is irreducible (since $G$ is connected). The eigenvalues if $\A(G)$ are denoted by $\alpha_1\le\alpha_2\le\dots\le\alpha_{n-1}<\alpha_n=\rho$. 
 The {\em Laplacian matrix} of $G$ is $L(G)=D(G) - \A(G)$, where  $D(G)$ is the diagonal matrix having the $i$th diagonal entry equal to the {\em degree} of the vertex $v_i$ (i.e., the number of edges incident with $v_i$). % The Laplacian matrix is {\em positive semidefinite}, meaning all eigenvalues are nonnegative. 
 The matrix  $Q(G)=D(G) + \A(G)$ is called the {\em signless Laplacian matrix}  $G$.  For a connected graph $G$ of order at least two, the normalized Laplacian is  $\nL(G)=\sqrt{D(G)}^{-1}L(G)\sqrt{D(G)}^{-1}$. %Definitions of matrices associated with digraphs are presented in Section \ref{s:dig}.
 
For $v_i,v_j\in V(G)$, the \emph{distance} between  $v_i$ and $v_j$, denoted by  $d(v_i,v_j)$, is the minimum length (number of edges) in a path starting at  $v_i$ and ending at $v_j$ (or vice versa). The {\em distance matrix} is  $\D(G)=[d(v_i,v_j)]$, i.e.,  the $n \times n$ matrix with $(i,j)$ entry equal to $d(v_i,v_j)$ \cite{GP71}. 
The {\em transmission} of vertex $v_i$ is  $t(v_i) =  \sum_{j=1}^n d(v_i,v_j)$.   
The {\em distance signless Laplacian} matrix and the {\em distance Laplacian} matrix are defined by  $\DQ(G) = T(G)  + \D(G)$  and $\DL(G) = T(G) - \D(G)$, where $T(G)$  is the diagonal matrix with $t(v_i)$  as the $i$-th diagonal entry \cite{AH13}.    The normalized distance Laplacian matrix is defined by $\nDL(G)=\sqrt{T(G)}^{-1}\DL(G)\sqrt{T(G)}^{-1}$ \cite{R20}. These matrices can be denoted by $\D,\DL,\DQ,\nDL$ when the intended graph is clear and $\D^*$ will be used  to denote a matrix that is one of (a subset of) these four matrices. 

The fact that $\DL(G), \DQ(G)$, and $\nDL(G)$ are {\em positive semidefinite} (meaning all eigenvalues are nonnegative) is well known and is discussed further in Section \ref{ss:PF}.  Observe that $\DQ(G)$ is positive and  $\D(G)$ is nonnegative and irreducible.  %Thus their {\em spectral radii} $\rho(\D(G))$  (i.e., eigenvalue of largest magnitude) is positive.  Furthermore, $\rho(\D(G))$ is a simple eigenvalue because $\D(G)$ is ; this fact is well known and 
Thus Perron-Frobenius theory applies, especially to the study of the {\em spectral radius}   (i.e., the largest magnitude of an eigenvalue, denoted by  $\rho(A)$ for  $A\in\Rnn$); this is discussed further in Section \ref{ss:PF}.

The eigenvalues of $\D$ are denoted by $\dev_1\le\dev_2\le\dots\le\dev_{n-1}<\dev_n=\rho(\D)$,  the eigenvalues of $\DQ$ are denoted by $\dqev_1\le\dqev_2\le\dots\le\dqev_{n-1}< \dqev_n=\rho(\DQ)$, the eigenvalues of $\DL$ are denoted by $0=\dlev_1<\dlev_2\le\dots\le\dlev_{n-1}\le \dlev_n=\rho(\DL)$, %\,\footnote{Note that in  \cite{AH13}  the eigenvalues of $\DQ$ and $\DL$ are numbered in decreasing order.} 
and the eigenvalues of $\nDL$ are denoted by $0=\ndlev_1<\ndlev_2\le\dots\le\ndlev_{n-1}\le \ndlev_n=\rho(\nDL)$.  Throughout the paper, we number these eigenvalues in increasing  (i.e., non-decreasing) order unless otherwise stated; in the literature, eigenvalues of real symmetric matrices are almost always labeled in order, but whether increasing or decreasing varies.

A (connected) graph is {\em transmission regular} or {\em $t$-transmission regular}  if every vertex has transmission $t$. In this case, the common value of the transmission of a vertex is denoted by $t(G)$. If a graph is transmission regular, then the distance matrix, distance signless Laplacian, distance Laplacian, normalized distance Laplacian are all  equivalent in the sense that any one can be derived from another by translation and scaling, so the eigenvalues of any one can be readily computed from those of another. %Then the adjacency matrix, Laplacian,  signless Laplacian, and normalized  Laplacian are all  equivalent in the sense that any one can be derived from another by translation ad scaling, so the eigenvalues of any one can be readily computed from those of another.  Similarly, the distance matrix, distance signless Laplacian, distance Laplacian, normalized distance Laplacian are all  equivalent.  
Specifically, if $G$ is $t$-transmission regular, then $\DQ(G)=t I_n+\D(G)$,  $\DL(G)=t I_n-\D(G), $ and $\nDL(G)=\frac 1 t\DL(G)$, so 
\beq\label{eq:tr-equiv}\dqev_i=t +\dev_{i},\ \dlev_i=t -\dev_{n+1-i}, \mbox{ and } \ndlev_i=1-\frac{\dev_{n+1-i}}{t }.\eeq   %A graph is {\em regular} or {\em $r$-regular} if every vertex has degree $r$.

 The {\em Wiener index} $W(G)$, which was introduced in chemistry  \cite{W47},
is  the sum of all distances between unordered pairs of vertices in $G$.  Observe that $W(G)$  is half the sum of all the entries in $\D(G)$, or equivalently,
\[W(G)=\frac 1 2 \sum_{i=1}^n t(v_i)=\frac 1 2 \tr \DL =\frac 1 2 \sum_{i=1}^n \dlev_i \mbox{ and } W(G)=\frac 1 2 \tr \DQ =\frac 1 2 \sum_{i=1}^n \dqev_i,\]
where $\tr M$ denotes the  trace of the matrix $M$.

The complete graph is the graph with all possible edges, i.e., $uv\in E(G)$ for all $u,v\in V(G)$. In a graph $G$, a path is a sequence of vertices $(v_{i_1},v_{i_2},\dots,v_{i_k})$ in which all vertices are unique and consecutive vertices are adjacent. A cycle is a sequence of vertices $(v_{i_1},v_{i_2},\dots,v_{i_k},v_{i_1})$ in which only the first/last vertex is repeated  and consecutive vertices are adjacent. For $n\ge 1$, the path graph $P_n$ is a graph with vertex set $V(P_n) = \{v_1, \dots, v_n\}$ and edge set $E(P_n) = \{v_1v_2, v_2v_3,\dots,v_{n-1}v_n\}$. For $n\ge 3$, a cycle graph $C_n$ is a graph with vertex set $V(C_n) = \{v_1, \dots, v_n\}$ and edge set $E(C_n) = \{v_1v_2, v_2v_3,\dots,v_{n-1}v_n,v_nv_1\}$. A {\em forest} is a graph that does not have cycles and a {\em tree} is a connected forest. The diameter of $G$, denoted $\diam(G)$, is the maximum distance between any two vertices in the graph.

%The following definitions and notations are standard and will be used throughout. 
Let $A$ be a $n\times n$ matrix. The {\em  characteristic polynomial of $A$} is $p_{A}(x)=\det(xI_n-A)$.   The {\em algebraic multiplicity} $\mult_A(z)$ of a number $z\in\C$ with respect to $A$ is the number of times $(x-z)$ appears as a factor in  $p_A(x)$ and its {\em geometric multiplicity} is the dimension of the eigenspace %$ES_M(z)$ 
 of $A$ relative to $z$. An eigenvalue is {\em simple} if its algebraic multiplicity is 1. %($\mult_A(z)=\gmult_A(z)=0$ if $z$ is not an eigenvalue of $M$).   
 The {\em spectrum} of $A$, denoted by $\spec(A)$,  is the multiset whose elements are the $n$ (complex) eigenvalues of $A$ (i.e., the number of times each eigenvalue appears in $\spec(A)$ is its algebraic multiplicity).  The eigenvalues are often denoted by $\lam_1,\lam_2,\dots, \lam_n$ (no order implied) and the spectrum is often written as $\spec(A)=\{\mu_1^{(m_1)},\dots,\mu_q^{(m_q)}\}$ where $\mu_1,\dots, \mu_q$ are the distinct eigenvalues of $M$ and $m_1,\dots,m_q$ are the (algebraic) multiplicities.
 
 In analogy with the generalized characteristic polynomial, Reinhart \cite{R20} introduced the distance generalized characteristic polynomial
 $\phi^{\D}(\lambda,r,G)=\det(\lambda I_n -\D(G)+rT(G))$ and showed that with the appropriate choices of parameters, it can provide the characteristic polynomials of $\D(G), \DQ(G), \DL(G)$, and $\nDL(G)$.

  A matrix $M=[m_{ij}]$ is symmetric if $m_{ij}=m_{ji}$ for all $i,j\in [n]$. For an $n\x n$ real symmetric matrix $M$, the eigenvalues of $M$ are real and the algebraic multiplicity and the geometric multiplicity are equal for each eigenvalue. A symmetric matrix $M$ has an  basis of eigenvectors and can be diagonalized, i.e. there exists an invertible matrix $S$ and a diagonal matrix $D$ such that $S^{-1}MS=D$.
Of particular interest here are the characteristic polynomials and spectra associated with $\D(G), \DL(G), \DQ(G),$ and $\nDL(G)$ for a graph $G$. In the case of graphs, these matrices are symmetric. However, the analogously defined matrices for digraphs need not be symmetric (see Section \ref{s:dig}).

 For an $n\times n$ matrix $A$, $A[X|Y]$ is the {\em submatrix} of $A$ with rows indexed by $X\subseteq [n]$ and columns indexed by $Y\subseteq [n]$. A {\em block matrix} is a matrix that can be viewed as being made up of submatrices defined by a partition of $[n]$. A matrix $A$ is {\em irreducible} if there does not exist a permutation matrix $P$ such that $P^{-1}AP$ is a block upper triangular matrix that has at least two blocks. A {\em nonnegative} matrix is a matrix whose entries are all nonnegative real numbers and a {\em positive} matrix is a matrix whose entries are all positive real numbers. The strongest results of Perron-Frobenius Theory are for positive matrices and irreducible nonnegative matrices, as discussed in Section \ref{ss:PF}. 
 
 Throughout the paper, we let $I$ or $I_n$ be the $n\times n$ identity matrix  and $J$ or $J_n$ be the $n\times n$ matrix of all ones. 
 The Kronecker product of matrices the $n\times n$ matrix $A=[a_{ij}]$ and the $n'\times n'$ matrix $A'$ is the $nn'\x nn'$ matrix $A\otimes A'=\begin{bmatrix}a_{11}A' &\dots& a_{1n}A'\\
 \vdots & \ddots & \vdots\\
 a_{n1}A' &\dots& a_{nn}MA\\
\end{bmatrix}$.

%%%%%%%%%%%%%%%%%%%%%%%%%%%
\section{Motivation for the distance matrix and early work}\label{s:D-loop-GL} % Leslie

The distance matrix of a graph was introduced by Graham and Pollak in \cite{GP71} to study the issue of routing messages or data between computers. That paper inspired much additional work on the distance matrix and more recently on variants such as the distance Laplacian matrix, leading  to many different research directions.
 In this section we discuss the motivating problem and early research on distance matrices.
%Graham and Lov\'asz \cite{GL78} studied the coefficients of the distance characteristic polynomial, including establishing a formula for the coefficients when the graph is a tree, and asked a variety of interesting questions that have been studied for years.

\subsection{Loop-switching and the distance matrix}\label{ss:D-loop}
As described by Graham and Pollak, for telephone calls in the late 60s and early 70s, it was reasonable to assume that the duration of the call was much longer than the time needed to find a circuit to connect the callers, so the circuit was normally established first. However such a model seems less suitable for transmitting information between computers.  They attribute to J.R.~Pierce a ``loop-switching'' model for a network as a sequence of connected one-way loops, including many small local loops, larger regional loops, and giant national loops.  The message will not have a pre-arranged route but at each junction must be able to readily determine whether to switch loops. %This determination is made possible by using an addressing scheme. 
Note that a configuration of  loops can be modeled as a graph, where each loop is a vertex, and two vertices are adjacent if and only if the corresponding loops intersect (see Figure \ref{f:loop}).

  \begin{figure}[!h]
\begin{center}
\scalebox{.4}{\includegraphics{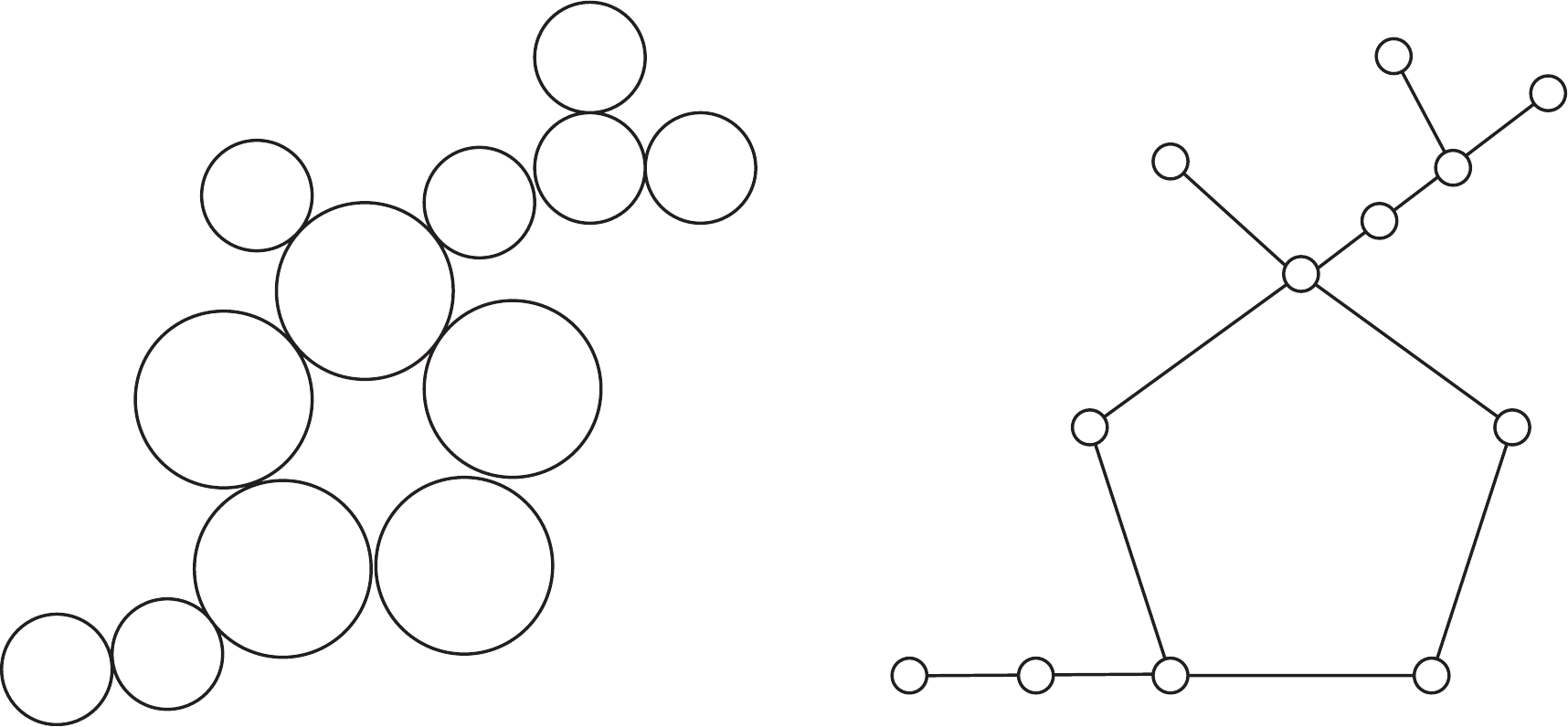}}
\caption{A loop configuration and its graph.  \label{f:loop}}
\vspace{-5pt}
\end{center}
\end{figure}

 Graham and Pollak proposed an addressing scheme to enable the message  to move efficiently  from its origin to its destination by switching when such change reduces the  `discrepancy' between the current address and the destination address. Perhaps the most natural such scheme would be to assign each loop a sequence of 0s and 1s as an address and to use the {\em Hamming distance}, i.e., the number of digits that differ, as a measure of discrepancy.  However this leads to difficulties (when, for example, the graph is a 3-cycle) and instead each loop/vertex $v$ is assigned an address  sequence $a(v)=a_1\dots a_r$ where $a_i\in\{0,1,*\}$ with $*$ neutral.  The {\em distance}  between addresses $a(v)$ and $a(v')=a'_1\dots a'_r$ is $d(a(v),a(v'))=|\{i:\{a_i,a'_i\}=\{0,1\}\}|$. The addresses are to be assigned such that $d(a(v),a(v'))=d(v,v')$, and it is not hard to see that this can always be done if $r$ is large enough.  This naturally raises question of the minimum value of $r$ needed for such an addressing scheme for graph $G$, denoted by $N(G)$.
 \bit
 \item[] For graphs $G$ of order $n$, is there an upper bound on $N(G)$ in terms of $n$?
 \eit
 
 Graham and Pollak proved that $N(G)\le \diam(G)(n-1)$ and presented an algorithm that will always produce a valid addressing scheme.  Their algorithm produced an address of length no more than $n-1$ for every graph of order $n$ to which they applied it  and they conjectured that $N(G)\le n-1$ \cite{GP71}; Winkler later established this conjecture:  %(Graham and Pollak also suggest that despite their mathematical interest in determining $N(G)$, in practice it is desirable to choose longer addresses if that simplifies the decision making at a junction.) 
 
 \begin{thm}{\rm \cite{W83}} If $G$ is a graph of order $n\ge 2$, then $N(G)\le n-1$.
 \end{thm}

Graham and Pollak \cite{GP71} showed this bound is tight by establishing the value of $N(G)$ for several families of graphs.

 \begin{thm}{\rm \cite{GP71}} %Let $G$ is a graph of order $n$, then $N(G)\le n-1$.
 $\null$
 \bit
 \item For $n\ge 2$, $N(K_n)=n-1$
 \item For $n\ge 3$, $N(C_n)=n-1$ if $n$ is odd and $N(C_n)=\frac n 2$ if $n$ is even.
 \item If $t$ is a tree of order $n\ge 2$, then $N(T)=n-1$.
 \eit
 \end{thm}

 For a graph $G$ of order $n$, let  $n_+(G)$ (respectively, $n_-(G)$) denote the number of positive (respectively, negative) eigenvalues of the distance matrix $\D(G)$, so the {\em inertia} of $\D(G)$ is the triple $(n_+(G),n_-(G),n-n_+(G)-n_-(G))$. Graham and Pollak established the next result, which they attribute to H.S. Witsenhausen.
 
 \begin{thm}{\rm\cite{GP71}} For a graph $G$, $N(G)\ge \max\{n_+(G),n_-(G)\}$.
 \end{thm}
 
 The seminal paper of Graham and Lov\'asz \cite{GL78} made a conjecture regarding the coefficients of the distance characteristic polynomial that was resolved only recently (this is discussed in the next section).  They also asked, 

 \bit
 \item[] Is there a graph $G$ for which $n_+(G)>n_-(G)$?
 \eit
 since all the examples for which the inertia had been determined at the time satisfied $n_+(G)\le n_-(G)$.
  Azarija exhibited a family of strongly regular graphs $G$ for which $n_+(G)>n_-(G)$ (see Theorem \ref{t:opt}).

%The initial work  of Graham and Pollack concerning  trees led to numerous additional results.

Graham and Pollak also established the value of the determinant of the distance matrix of a tree.

\begin{thm}{\rm \cite{GP71}} If $T$ is a graph of order $n\ge 2$, then $\det \D(T)=(-1)^{n-1}(n-1)2^{n-2}$.  Furthermore, $n_+(T)=1$ and $n_-(T)=n-1$.
 \end{thm}

The tree-determinant result was extended to the characteristic polynomial for trees (as described in the next section) and to arbitrary graphs for the determinant. Graham, Hoffman, and Hosoya showed that the determinant the distance matrix of a graph depends only on the determinants and cofactors of its blocks.   A {\em block} of a graph  is a subgraph that has no cut vertices and is maximal with respect to this property.  A graph is the union of its blocks.  Observe that in a tree all blocks have order two, and a tree of order $n$ has $n-1$ blocks. The next result was proved for strongly connected directed graphs, but since we have yet not defined these terms we state it for (connected) graphs.

 \begin{thm}{\rm \cite{GHH77}} Let $G$ be a graph with blocks $G_1, G_2,\dots,G_k$. Then 
 \[\det \D(G)=\ds\sum_{i=1}^k\det(\D(G_i))\prod_{j\ne i}cof(\D(G_j))  \]
 where $cof(M)$ is the sum of all the cofactors of $M$. \end{thm}

\subsection{Trees and  the Graham and Lov\'asz Unimodality Conjecture}\label{ss:D-GL}
%\cite{GL78}

%Let $G$ be a  graph  on $n$ vertices.  
The {\em distance characteristic polynomial of $G$} or {\em distance polynomial of $G$} is $p_{\D}(x)=\det(xI_n-\D(G))$. 
In much of the initial work, including \cite{EGG76, GL78}  the polynomial studied is $\Delta_G(x)=\det(\D(G)-xI_n)=(-1)^np_{\D}(x)$ where  $n$ is the order of the graph $G$ (in these papers, $\Delta_G(x)$ was called the distance characteristic polynomial of $G$).  The  coefficient of $x^k$ in $\Delta_G(x)$ is denoted  by $\delta_k(G)$ %or $\delta_k$ 
\cite{GL78}. Thus the coefficient of $x^k$ in the distance polynomial of $G$ is $(-1)^n\delta_k(G)$.  
Edelberg, Garey and Graham computed some coefficients of $\Delta_G(x)$ and determined the sign of each coefficient $\delta_k$.  
 \begin{thm}\label{t:egg}{\rm\cite{EGG76}} For a tree $T$ on $n$ vertices,
 \[\sgn(\delta_k(T))= \begin{cases}
(-1)^n & \mbox{for }k=n \\
0 & \mbox{for }k=n\\
(-1)^{n-1} & \mbox{for } 0\le k\le n-2
\end{cases}. \]
 \end{thm}

As a consequence,    $(-1)^{n-1}\delta_k(T) > 0$ for $0\le k\le n-2$, as noted in \cite{GL78}, so the coefficient of $x^k$ in $p_{\D(T)}(x)$ is negative for $0\le k\le n-2$.
 Graham and Lov\'asz extended the work of Edelberg, Garey, and Graham, showing that the coefficients of the distance characteristic polynomial of a tree depend only on the number of occurrences of subforests of the tree. Let $N_F(T)$ denote the number of occurrences of $F$ in $T$ (with $N_F(T)=1$ if $F$ has order zero).

 \begin{thm}{\rm\cite{GL78}} For a tree $T$ on $n\ge 2$ vertices, 
 \[\delta_k(T)=(-1)^{n-1}2^{n-k-2}\sum_FA^{(k)}_F N_F(T)\]
 where $F$ ranges over forests having $k-1,k,k+1$ edges and no isolated vertices and $A^{(k)}_F$ is an integer that depends on the number of occurrences of various paths in $F$ but does not depend on $T$.
 \end{thm}
 
 For a graph $G$ of order $n$ and $0\le k\le n-2$, define $ d_k(G)= |\delta_k(G)|/2^{n-k-2}$. The numbers $d_k(G)$ are called the {\em normalized coefficients}.
A sequence $a_0,a_1,a_2,\ldots, a_n$ of real numbers is {\em unimodal} if there is a $k$ such that $a_{i-1}\leq a_{i}$ for $i\leq k$ and $a_{i}\geq a_{i+1}$ for $i\geq k$.  Graham and Lov\'asz made 
the following statement  in \cite{GL78}, which has come to called the {\em Graham-Lov\'asz Conjecture}: 

\begin{quote}{\em  It appears that in fact for each tree $T$, the quantities $(-1)^{n-1}\delta_k(T)/2^{n-k-2}$ are {unimodal}  with the maximum value occurring for $k=\big\lfloor \frac n 2\big\rfloor$.  We see no way to prove this, however.
}\end{quote}

The conjecture can be be restated as follows.
 \begin{conj} [Graham-Lov\'asz]\label{GLconj}
For a tree $T$ of order $n\ge 3$, the sequence of normalized coefficients  $d_0(T),\dots,d_{n-2}(T)$ is unimodal and  the peak occurs at $\big\lfloor \frac n 2\big\rfloor$.\end{conj}

 The  location of the peak as stated in Conjecture \ref{GLconj}  was disproved by Collins in 1985.\footnote{Despite use of the term {\em coefficient}  throughout \cite{C89}, the sequence discussed there  is $d_k(T)$, not $\delta_k(T)$.}

\begin{thm}{\rm \cite{C89}} For both stars and paths the sequence $d_0(T),\dots,d_{n-2}(T)$ is unimodal, but for paths the peak is at approximately $\left(1-\frac 1 {\sqrt 5}\right)n$ (for stars it is  at $\left\lfloor \frac n 2\right\rfloor$). 
\end{thm}

 Collins attributes the next version of the conjecture to Peter Shor: 

\begin{conj}{\rm \cite{C89}}\label{CSconj}\label{peakconj} The normalized coefficients  of the distance characteristic polynomial for any tree $T$ with $n\ge 3$ vertices are unimodal  with peak between $\big\lfloor \frac n 2\big\rfloor$ and $\left\lceil \left(1-\frac 1 {\sqrt 5}\right)n\right\rceil$.
\end{conj}

For the rest of this section, the order of a graph is assumed to be at least three (any sequence $a_0$ is trivially unimodal and the peak location is 0). 
The unimodality of the normalized coefficients was established in 2015 using results concerning coefficients of polynomials having all roots real and log concavity.  A sequence $a_0,a_1,a_2,\ldots, a_n$ of real numbers  is {\em log-concave} if $a_j^2\geq a_{j-1}a_{j+1}$ for all $j=1,\dots,  n-1$. 

\begin{thm}{\rm \cite{GRWC15-uni}} Let $T$ be  a tree  of order $n\ge 3$.
\bit
\item The  coefficient sequence  of the distance characteristic polynomial  of $T$, $p_{\D}(x)$, is log-concave.
\item The sequence $|\delta_0(T)|,\dots,|\delta_{n-2}(T)|$ of absolute values  of coefficients of the distance characteristic polynomial  is log-concave and unimodal.
\item The sequence $d_0(T),\dots,d_{n-2}(T)$ of normalized coefficients  of the distance characteristic polynomial    is log-concave and unimodal.
\eit
\end{thm}

Bounds on the location of the peak for trees were presented (that paper also includes a more refined upper bound than the one listed next that depends on the structure of the tree).

\begin{thm}{\rm \cite{GRWC15-uni}} 
Let $T$ be a tree on $n \ge 3$ vertices with diameter $d$. The peak location of the normalized coefficients $d_0(T),d_1(T),\ldots ,d_{n-2}(T)$  is at most $\lc\frac{2}{3}n\rc$  and is at least $\left\lfloor\frac{n-2}{1+d}\right\rfloor$.
\end{thm}

It was also shown in \cite{GRWC15-uni} that the sequence $d_0(H),\dots,d_{12}(H)$ of normalized coefficients is not unimodal for the Heawood graph $H$.

Throughout the prior discussion, only the distance matrix has been considered.  However,  the unimodality of the coefficients of the distance Laplacian characteristic polynomial was established recently \cite{GRWC18} and next we  establish the unimodality of the coefficients of the distance signless Laplacian and normalized distance Laplacian  characteristic polynomials.
The {\em distance signless Laplacian characteristic polynomial of $G$}  is $p_{\DQ}(x)=\det(xI_n-\DQ(G))$, the {\em distance Laplacian characteristic polynomial of $G$}  is $p_{\DL}(x)=\det(xI_n-\DL(G))$, and the {\em normalized distance Laplacian characteristic polynomial of $G$}  is $p_{\nDL}(x)=\det(xI_n-\nDL(G))$. 

\begin{thm}{\rm \cite{GRWC18}} Let $G$ be a graph of order $n$, and let $p_{\DL}(x)=x^n+\delta^L_{n-1}x^{n-1}+\dots+\delta^L_{1}x$.  Then the sequence $\delta^L_{1},\dots,\delta^L_{n-1},\delta^L_{n}=1$ is log-concave and $|\delta^L_{1}|,\dots,|\delta^L_{n}|$ is unimodal.  In fact, $|\delta^L_{1}|\geq \dots \geq |\delta^L_{n}|$. 
\end{thm} 

 Next we show that unimodality extends to any positive semidefinite matrix, using  the method from \cite{GRWC18}.

\begin{thm}\label{t:nonneg-unimod} Let $M$ be a positive semidefinite matrix and let $p_M(x)=x^n+m_{n-1}x^{n-1}+\dots+m_{1}x+m_0$.  Then the sequence $m_{0},\dots,m_{n-1},m_{n}=1$ is log-concave and the sequence $|m_{0}|,\dots,|m_{n}|$ is unimodal.  

Therefore, coefficients of the distance signless Laplacian characteristic polynomial (respectively, the normalized distance Laplacian characteristic polynomial) are log-concave, and the absolute values of these coefficients are unimodal.     \end{thm} 
\bpf
 It is known that  the  coefficient sequence $a_0,a_1,\dots, a_{n-1},a_n=1$ of the characteristic polynomial of any real symmetric matrix is log-concave, and if  all entries of a subsequence  $a_s,\dots,a_t$  alternate in sign, then the subsequence of absolute values $|a_s|,  \dots,  |a_t|$ is unimodal \cite{GRWC18}.  Denote the eigenvalues of $M$ by $\mu_1\le\dots\le \mu_n$ and let $c$ denote the multiplicity of eigenvalue zero  for $M$ (with $c=0$ signifying zero is not an eigenvalue of $M$). Note that $m_{k}=(-1)^{n-k} S_{n-k}(\mu_1,\mu_2,\dots,\mu_n)$ where \[S_{k}(b_1,b_2,\dots,b_n)=\ds\sum_{S\subseteq [n],|S|=k}\prod_{i\in S}b_i  \] is the $k$th symmetric function of $b_1,b_2,\dots,b_n$.  Since $\mu_1=\dots=\mu_c=0<\mu_{c+1}\le \cdots\le \mu_n$, $S_{k}(\mu_1,\mu_2,\dots,\mu_n)>0$ for $k\le n-c$ and  $S_{k}(\mu_1,\mu_2,\dots,\mu_n)=0$ for $k>n-c$. Thus $|m_k|>0$ for $k\ge c$ and all the nonzero coefficients alternate in sign. This implies the sequence $\{|m_k|\}_{k=0}^{n}$ is log-concave and $\{|m_k|\}_{k=c}^{n}$ is unimodal.  Since $m_k=0$ for $k<c$, $|m_{0}|,\dots,|m_{n}|$ is unimodal.  %Thus the  coefficient sequence $m_0,m_1,\dots, m_{n-1},m_n=1$  is log-concave and the subsequence of absolute values $|m_0|,  \dots,  |m_n|$ is unimodal.   
 As noted in the introduction, the distance signless Laplacian matrix and normalized distance Laplacian matrix of a graph are positive semidefinite. 
%Now suppose $\mu_{c+1} \geq n-c$. Since $\{|m_k|\}_{k=c}^{n}$ is unimodal and $m_k=0$ for $k<c$, to show that $|m_c|\geq \dots \geq |m_n|$ it is sufficient to show that $|m_c| > |m_{c+1}|$. \[ |m_{c+1}| = \sum\limits_{i=c+1}^n \prod\limits_{j\in\{c+1,\dots,n\}, j \neq i} \mu_j \leq (n-c) \prod\limits_{j=c+2}^{n} \mu_j < \prod\limits_{j=2}^n \mu_j = |m_c|. \qedhere \]
 \epf

The result $|\delta^L_{1}|\geq \dots \geq |\delta^L_{n}|$ can be extended to show that if $|\mu_{c+1}|>n-c$ (where $\mu_i$ is the $i$th eigenvalue of $M$ and $c$ is the multiplicity of zero), 
the nonzero coefficients decrease with increasing index, so the peak of the unimodal sequence $|m_{0}|,\dots,|m_{n}|$ is at $k=c$ (and $m_k=0$ for $k<c$). However, for  the distance signless Laplacian matrix or the normalized distance Laplacian matrix (even for trees), the hypothesis $|\mu_{c+1}|>n-c$ fails and the peak need not be located at $k=c$, as the next example shows. 

\begin{ex} For $K_{1,3}$, $p_{\DQ}(x)=x^4 - 18x^3 + 105x^2 - 252x + 216$ and $p_{\nDL}(x)=x^4 - 4x^3 + 5.32x^2 - 2.352x$.  For comparison, $p_{\DL}(x)=x^4 - 18x^3 + 105x^2 - 196x$ for $K_{1,3}$.
\end{ex}

%If  $|\mu_{c+1}|>n-c$ where  $c$ is the multiplicity of eigenvalue zero  for $M$ (with $c=0$ signifying zero is not an eigenvalue of $M$), then the peak of the unimodal sequence $|m_{0}|,\dots,|m_{n}|$ is at $k=c$ (and $m_i=0$ for $0\le i \le c-1$). 

%%%%%%%%%%%%%%%%%%%%%%%%%%%
\section{Techniques for computing spectra}\label{s:method} % Leslie

In this section, we describe some techniques that have been used to compute spectra of various types of distance matrices.
Many methods use eigenvectors, which are particularly effective since every real symmetric matrix has a basis of eigenvectors. %to relate the eigenvalues of $ \DL(G), \DQ(G)$, and $\nDL(G)$ to those of $\D(G)$;  . 

%------------------------------------------------------------------
\subsection{Twins and quotient matrices}\label{ss:quot}

Let $v_1, v_2$ be vertices of a graph $G$ of order at least three  that  have the same neighbors other than $v_1$ and $v_2$.  
 If $N[v_1]=N[v_2]$ (so $v_1$ and $v_2$ are adjacent), then they are   called {\em adjacent twins}. If $N(v_1)=N(v_2)$ (so $v_1$ and $v_2$ are not adjacent), then they are called {\em independent twins}.  Both cases are referred to as {\em twins}.  Note that twins have the same transmission and are at distance one (adjacent twins) or two (independent twins) from each other.  Observe that if  $v_k$ and $v_{k+i}$ are twins for $i=1,\dots,r-1$, then for $i\ne j\in\{1,\dots,r-1\}$,  $v_{k+i}$ and $v_{k+j}$ are twins of the same type as $v_k$ and $v_{k+i}$, because $N[v_{k+i}]=N[v_k]=N[v_{k+j}]$ for adjacent twins and  $N(v_{k+i})=N(v_k)=N(v_{k+j})$ for independent twins.
 
 It is useful to partition the vertices with one or more partition sets consisting of twins  and to use the partition to create block matrices, as in the proofs of Theorems \ref{t:twin} and \ref{p:twin-quot}.  If $M=[m_{ij}]$ is an $n\x n$ matrix,  $X=(X_1,\dots,  X_p)$ is a partition of $[n]$ with each set $X_i$ consisting of consecutive integers, then the partition $X$ defines a $p\x p$ block matrix $[M_{ij}]$ where $M_{ij}=M[X_i|X_j]$ (one can define a block matrix without the assumption that each  partition set  consists of consecutive integers, but it is notationally  simpler to relabel the graph to achieve the consecutive property). 
 
\begin{thm}\label{t:twin} Let $G$ be a graph of order at least three, let $t=t(v_k)$, and suppose that $v_k$ and $v_{k+i}$ are twins for $i=1,\dots,r-1$.  Then   $[0,\dots,1 , 0,\dots,0,-1 ,0 , \dots , 0]^T$ (the $k$th coordinate is $1$ and the $k+i$th coordinate is $-1$) is an eigenvector for each matrix and eigenvalue $\lam$ listed below for $i=1,\dots,r-1$.  Thus $\lam$ has multiplicity at least $r-1$.
\ben[$(1)$]
\item\label{twin-D}  $\D(G):$  $\lam=-2$ if $v_k$ and $v_{k+i}$ are independent;   $\lam=-1$ if $v_k$ and $v_{k+i}$ are adjacent.  
\item\label{twin-DQ}  $\DQ(G):$    $\lam=t-2$ if $v_k$ and $v_{k+i}$ are independent;   $\lam=t-1$ if $v_k$ and $v_{k+i}$ are adjacent.
\item\label{twin-DL} $\DL(G):$    $\lam=t+2$ if $v_k$ and $v_{k+i}$ are independent {\rm \cite{GRWC18}};   $\lam=t+1$ if $v_k$ and $v_{k+i}$ are adjacent.
\item{\rm \cite{R20}}  $\nDL(G):$    $\lam=\frac{t+2}t$  if $v_k$ and $v_{k+i}$ are independent;   $\lam=\frac{t+1}t$ if $v_k$ and $v_{k+i}$ are adjacent. 
\een
%In each case, a set of $r$ twins generates $r-1$ independent eigenvectors and thus $r-1$ copies of the associated eigenvalue. 
 \end{thm}
 \bpf 
 The method used in \cite{GRWC18} to prove \eqref{twin-DL} for independent twins can be used to establish the other eigenvector results. Here we show  $\bw=[1,-1 ,0 , \dots , 0]^T$ is an eigenvector for eigenvalue $\lam=t-1$ of $\DQ(G)$ where $v_1$ and $v_{2}$ are adjacent twins (the argument is the same for $v_k$ and $v_{k+i}$ but the notation is messier).  The remaining cases are similar. 
 
 Apply the partition $\{1,2\}, \{3,\dots,n\}$ to $\DQ(G)$ and $\bw$ to define block matrices and multiply:\vspace{-5pt}
\[\DQ(G)\bw=\mtx{D_{1,1}& D_{2,1}^T\\D_{2,1}&D_{2,2}}\mtx{\bw_1\\ \bw_2}=\mtx{D_{1,1}\bw_1+D_{1,2} \bw_2\\D_{2,1}\bw_1+D_{2,2} \bw_2}\vspace{-5pt}\]
Since $D_{1,1}=\mtx{t & 1 \\1 & t}$, $D_{2,1}=\mtx{\bd & \bd}$ for some vector $\bd$,  $\bw_1=\mtx{1\\-1}$, and $\bw_2=\bzero$,  \vspace{-5pt}
\[D_{1,1}\bw_1+D_{1,2} \bw_2=\mtx{t-1\\1-t}+\bzero_2=(t-1)\bw_1\mbox{ and }
D_{2,1}\bw_1+D_{1,2} \bw_2=\bzero_{n-2}+\bzero_{n-2}=(t-1)\bw_2.\vspace{-5pt} \]
Thus $\DQ(G)\bw=(t-1)\bw$. 
\epf

Quotient matrices are an important tool in the study of distance matrices (see, for example, \cite{AP15}). Let $M=[m_{ij}]$ be a symmetric $n\x n$ matrix, let $X=(X_1,\dots,  X_p)$ be a partition of $[n]$ with each set $X_i$ consisting of consecutive integers, and let $n_i=|X_i|$ for $i=1,\dots,p$.  %The partition $X$ defines a block matrix $[A_{i,j}]$ where $A_{i,j}=A[X_i|X_j]$. 
The {\em quotient matrix} $B=[b_{ij}]$ of $M$ for this partition is the $p\x p$ matrix with entry $b_{ij}$ equal to the average row sum of the submatrix $M_{ij}=M[X_i|X_j]$. The partition $X$ is {\em equitable} for $M$ if for every pair $i,j\in\{1,\dots,p\}$, the row sums of $M_{ij}$ are constant. % If $X$ is an equitable partition for $A$, then .
The {\em characteristic matrix} of $X$ is the $n\x p$ matrix $S=[s_{ij}]$  defined by  $s_{ij}=1$ if $i\in X_j$ and $s_{ij}=0$ if $i\not\in X_j$. 

\begin{lem}\label{t:quot}  
Let $M$ be a symmetric $n\x n$  matrix, let $X$ be an equitable partition $X$ of $[n]$,  let $B$ be the quotient matrix  of $M$ for  $X$, and let $\bx,\by,\bz\in \R^p$.  
\ben[$(1)$]
\item\label{l:quot-0} {\rm \cite[p. 24]{BH}} $MS=SB$.
\item\label{l:quot-1} If $i\in X_j$, then $(S\bx)_i=x_j$ where $(S\bx)_i$ denotes the $i$th coordinate of $S\bx$ and $x_j$ denotes the $j$th coordinate of $\bx$.
\item\label{l:quot-2} If $S\bx=S\by$, then $\bx=\by$.
\item\label{l:quot-3} If $S\bz$ is an eigenvector of $M$, then  $\bz$ is an eigenvector of $B$ for the same eigenvalue.
\item\label{l:quot-4} {\rm \cite[Lemmas 2.3.1]{BH}} If $\bz$ is an eigenvector of $B$, then  $S\bz$ is an eigenvector of $M$ for the same eigenvalue.  \een\end{lem}
\bpf It is straightforward to verify \eqref{l:quot-1}, and \eqref{l:quot-1} implies \eqref{l:quot-2}. For \eqref{l:quot-3}, assume $S\bz$ is an eigenvector of $M$ for $\mu$.  Then
$S(\mu\bz)=\mu S\bz=MS\bz=SB\bz=S(B\bz)$.  Then $\mu\bz=B\bz$ by \eqref{l:quot-2}.
\epf

Sets of twins in a graph naturally provide an equitable partition of any of the four variants of the distance matrix.  Theorem  \ref{t:twin} and Lemma \ref{t:quot} can be combined to determine the spectrum.  We use $\D^*$ to denote one of $\D(G), \DL(G), \DQ(G)$, or $\nDL(G)$. 

\begin{thm}\label{p:twin-quot} Let $X=(X_1,\dots,X_p)$ be a partition of the vertices of $G$ with $n_1\le \dots\le n_p$ and let $k$ be the least index such that $n_k\ge 2$.  Suppose that $v,u\in X_j$ implies $v=u$ or $v$ and $u$ are twins.  For $j=k,\dots,m$, let  $\lam_j$ denote the the eigenvalue $\lam$ specified in Theorem \ref{t:twin} for $\D^*$ and the type of twin in $X_j$. Let $B$ denote the quotient matrix of $\D^*$ for $X$.  Then  $\spec(D^*)=\{\lam_k^{(n_k-1)},\dots,\lam_p^{(n_p-1)}\}\cup\spec(B)$ (as multisets).
\end{thm}
\bpf     
Apply Theorem  \ref{t:twin} to construct $n_j-1$ eigenvectors for $\lam_j$, $j=k,\dots,m$ and denote this entire collection of  eigenvectors by  $\bw_1,\dots,\bw_{n-m}$; let $W_j$ denoted the span of the subset of these vectors that are associated with $\lam_j$.  It is immediate that $\{\lam_k^{(n_k-1)},\dots,\lam_p^{(n_p-1)}\}\subset\spec(\D^*)$ (as multisets). By Lemma \ref{t:quot}, every eigenvector $\bz$ of $B$ for eigenvalue $\mu$ yields an eigenvector $S\bz$ of $\D^*$ for $\mu$.  Furthermore, $S\bz$ is orthogonal to (and thus independent of) $\bw_1,\dots,\bw_{n-m}$. Hence it suffices to show that $B$ has a basis of eigenvectors.

 Extend $\{\bw_1,\dots,\bw_{n-m}\}$ to a basis of eigenvectors $\{\bw_1,\dots,\bw_{n-m},\bw_{n-m+1},\dots,\bw_{n}\}$ of $\D^*$ (a basis of eigenvectors exists because $\D^*$ is symmetric). Consider $\bw_h$ with $h>n-m$.  If the associated eigenvalue $\mu_h$ of $\D^*$ is distinct from $\lam_j$, then $\bw_h$ is orthogonal to the eigenvectors for $\lam_j$.  If $\mu_h=\lam_j$, then let $\bw'_h=\bw_h-\proj_{W_j}(\bw_h)$ (this step can be applied more than once if needed).  Then $\bw'_h$ is an eigenvector for  $\mu_h$ and is orthogonal to  $\bw_\ell$ for   $\ell=1,\dots,n-m$.  This implies $\bw'_h$ is constant on 
the  coordinates in $X_j$ for $j=1,\dots,m$, so $\bw'_h=S\bz_h$ for some $m$-vector $\bz$.  By Lemma \ref{t:quot}, $\bz_h$ is an eigenvector for $B$ for $\mu_h$. Thus  $B$ has a basis of eigenvectors and $\spec(D^*)=\{\lam_k^{(n_k-1)},\dots,$ $\lam_p^{(n_p-1)}\}\cup \spec(B)$  (as multisets). \epf

The use of Theorem \ref{p:twin-quot} is illustrated in the proof of Proposition \ref{p:Sn+e}.

%------------------------------------------------------------------
\subsection{Graphs of diameter 2}\label{s:diam2}

 Sometimes we can relate the distance matrix $\D(G)$ to the adjacency matrix $\A(G)$, as is the case for 
 a graph $G$ of diameter at most $2$. If $\diam(G)\le 2$, then any pair of nonadjacent vertices has distance two, so $\D(G)=2(J-I)-\A(G)$.  
 This is most useful when $G$ is transmission regular.    The results described here are well known.
 
 \begin{rem}\label{r:commJ} Let $M$ be a real symmetric matrix with all row sums equal to $r$. Then $M$ commutes with the all ones matrix $J$. Furthermore, the vector $\dsone=[1,\dots,1]^T$ is a common eigenvector of $J$ and $M$. Since $M$ is symmetric, $M$ has an orthogonal basis of eigenvectors $\{\dsone,\bx_2,\dots, \bx_n\}$ where $M\bx_i=\lam_i\bx_i$ but  it is not assumed that $\lam_i\le \lam_{i+1}$. Since $\bx_i$ is orthogonal to $\dsone$ for $i=2,\dots,n$, $\bx_i$  is an eigenvector of $J$ for eigenvalue $0$.  
\end{rem}

 \begin{rem}\label{r:diam2} Suppose   $\diam(G)\le 2$. Then $G$ is  $t$-transmission regular if and only if $G$ is $k$-regular with  $k+t=2n-2$. Since $\A(G)$ is nonnegative and irreducible, $\rho(\A(G))=k$ is a simple eigenvalue (see Section \ref{ss:PF}). Since $\A(G)$ commutes with $J$, %\beq\label{eq:tr-diam2}
\[ \dspec(G)=\{t=2n-2-k\}\cup\{-\alpha_i-2:i=1,\dots,n-1\}\] where the eigenvalues of $\A(G)$ are $\alpha_1\le\dots\le\alpha_{n-1}<k$. Then from \eqref{eq:tr-equiv},\\
$\dQspec(G)=\{2t\}\cup\{t-\alpha_i-2:i=1,\dots,n-1\},$ and $\ndLspec(G)=\{0\}\cup\{1+\frac{\alpha_i-2}t:i=1,\dots,n-1\}, \ \dLspec(G)=\{0\}\cup\{t+\alpha_i+2:i=1,\dots,n-1\}.$  %A graph is {\em regular} or {\em $r$-regular} if every vertex has degree $r$.

\end{rem}

The derivation of the distance eigenvalues has also been done by other methods in  \cite{Az14} and  \cite{AP15}. % and the methods in Remark \ref{r:diam2} have been used by various authors to determine spectra of one or more of these four matrices.

%------------------------------------------------------------------
\subsection{Spectra of products of graphs}\label{s:prod}

In this section we summarize results for distance spectra of graphs, which can be easily applied to the distance signless Laplacian, distance Laplacian, and normalized distance Laplacian via equation \eqref{eq:tr-equiv} for transmission regular graphs.  Analogous results are known for adjacency spectra of regular graphs. The {\em Cartesian product} of two graphs  $G=(V,E)$ and $G'=(V',E')$ is the graph $G\cp G'$, the graph whose vertex set is the Cartesian product $V\x V'$ and where two vertices $(u,u')$ and $(v,v')$ are adjacent if ($u=v$ and $\{u',v'\}\in {E'}$) or ($u'=v'$ and $\{u,v\}\in E$). 

Matrix products play a key role is establishing results for products.  
The next theorem appears in \cite{I2009}, where it is stated for distance regular graphs, but as noted in \cite{AP15}, the proof applies to transmission regular graphs.  It is proved by showing that with a suitable ordering of vertices, \[\D(G\cp G')=\D(G)\otimes J_{n'}+J_n\otimes \D(G').\]  

\begin{thm}\label{tr-dspec-cprod}{\rm \cite{I2009}} Let $G$ and $G'$ be transmission regular graphs of orders $n$ and $n'$, respectively. Let $t=t(G)$, $t'=t(G')$, $\dspec(G)=\{ \dev_1,\dots,\dev_{n-1},\dev_n=t\}$ and $\dspec(G)=\{ \dev'_1,\dots,\dev'_{n-1},\dev'_n=t'\}$.  Then
 \[\dspec(G \cp G')=\{ n' k+n t'\}\cup\{ n' \dev_1,\dots,n'\dev_{n-1}\}\cup\{ n \dev'_1,\dots,n\dev'_{n'-1}\} \cup \{0^{((n-1)(n'-1))}\}.\] \end{thm}

Theorem \ref{tr-dspec-cprod} and equation \eqref{eq:tr-equiv} provide a method to establish the  spectra of the hypercube and other Hamming graphs for the four types of distance matrices; see Proposition \ref{t:Ham}. 

The {\em lexicographic product} of $G$ and $G'$ is the graph $G\lexp G'$, the graph whose vertex set is $V\x V'$ and where two vertices $(u,u')$ and $(v,v')$ are adjacent if $(u,v)\in E(G)$ or ($u=v$ and $\{u',v'\}\in E$). Note that in the next theorem, $G$ need not be transmission regular. However, when $G$ is $t$-transmission regular, $G \lexp G'$ is $(tn'+2n-2-r')$-transmission regular and so equation \eqref{eq:tr-equiv} applies. The result is proved by showing that with a suitable ordering of vertices, 
\[\D(G\lexp G')=\D(G)\otimes J_{n'}+I_n\otimes (\A(G')+2\A(\ol{G'})).\]  
 
\begin{thm}\label{dspec-lexprod}{\rm \cite{I2009}} Let $G$ and $G'$ be graphs of orders $n$ and $n'$, respectively, and assume $G'$ is $k$-regular.  Let $\dspec(G)=\{ \dev_1,\dots,\dev_{n-1},\dev_n\}$ and $\spec(\A(G'))=\{ \alpha'_1,\dots,\alpha'_{n-1},\alpha'_n=k'\}$.  Then 
{\rm \[\dspec(G \lexp G')=\{ n' \dev_i+2n'-k'-2:i=1,\dots,n\}\cup\{ (-\alpha'_j-2)^{(n))}:j-1,\dots,n'-1\}.\]} \end{thm}

%------------------------------------------------------------------
\subsection{Additional linear algebraic techniques}\label{ss:PF}

In this section we briefly highlight the use of several additional results from linear algebra  in the study of spectra of various distance matrices.  \ms

 {\bf Ger\v sgorin disks}
Let $A=[a_{ij}]$ be an $n\x n$ complex matrix (symmetry is not required).   In general, the eigenvalues of $A$ may be complex (nonreal) even if $A$ is real.  The Ger\v sgorin Disk Theorem describes a region of the complex plane that contains all the eigenvalues of $A$
 \cite[Theorem 6.1.1]{HJ}): The $i$th {\em punctured absolute row sum} of $A$ is  $r'_i(A)=\sum_{j\ne i}|a_{ij}|$.  Define the $i$th {\em Ger\v sgorin disk}  of $A$  to be the set  of all complex numbers within the circle in the complex plane  of radius $r'_i(A)$ centered at $a_{ii}$.  Then the union of these disks contains the spectrum of $A$.  That is, 
\[ \spec(A)\subseteq  \cup_{i=1}^nD_i(A) \mbox{ where } D_i(A)=\{\zeta : |\zeta-a_{ii}|\le r'_i(A)\}.\]
%The union of all the Gershgorin disks of $A$ will be denoted by $D_A$.
We have an upper bound on the spectral radius as an immediate consequence of the Ger\v sgorin Disk Theorem: $\rho(A)\le  \max_{i=1}^n \sum_{j=1}^n\left|a_{ij}\right|$.

 For a symmetric real matrix, the eigenvalues are real and the  Ger\v sgorin  disks are restricted to real intervals.  
For an $n\x n$ real symmetric $M$, 
$\spec(M)\subset\cup_{i=1}^n[m_{ii}-r'_i(M),\, m_{ii}+r'_i(M)].$
Since the $i$th diagonal entry of $\DL(G)$ or $\DQ(G)$  is $t(v_i)$ and $r'_i(\DL(G))=r'_i(\DQ(G))=t(v_i)$, this implies  $\DL(G)$ and $\DQ(G)$ are positive semidefinite.\ms

{\bf Perron-Frobenius theory} There is an extensive theory of spectra of nonnegative matrices, called Perron-Frobenius theory. %, which offers many useful tools for study of spectra when applicable.  
The distance matrix and the distance signless Laplacian matrix of a graph are nonnegative, so Perron-Frobenius theory applies. Here we mention only results that have been applied to finding  the spectra of one or both of these matrices for graphs or digraphs (we do not restrict the discussion to symmetric matrices).  A more extensive treatment, that includes all the results here, can be found in \cite[Chapter 8]{HJ}.
 Let $A=[a_{ij}]$ be an $n\x n$ nonnegative matrix.  The   $i$th {\em row sum} of $A$ is  $r_i(A)=\sum_{j=1}^na_{ij}$.  
Observe that $r_i(\D(G))=t(v_i)$ and $r_i(\DQ(G))=2t(v_i)$.  One well known result is 
\[\min_{i=1}^n r_i(A)\le\rho(A)\le  \max_{i=1}^n r_i(A).\]
The strongest results are for positive matrices and irreducible matrices; $\DQ(G)$ is  positive and $\D(G)$ is irreducible.   Here we state parts of the Perron's Theorem and the Perron-Frobenius Theorem (see \cite[Theorems 8.2.8 and 8.4.4]{HJ}).  
 Suppose $A$ is an $n\x n$ irreducible nonnegative matrix with $n\ge 2$.  Then\vspace{-5pt}
\bit
\item $\rho(A)>0$ is a simple eigenvalue with a positive eigenvector.\vspace{-5pt}
\item If in addition $A$ is positive, then $\rho(A)>|\lam|$ for every eigenvalue $\lam\ne \rho(A)$.
\eit
%\ms

 {\bf Rayleigh quotients}
 Let $M$ be an $n\x n$ real symmetric matrix.  For $n$-vector $\bx$, $\frac{\bx^TM\bx}{\bx^T\bx}$ is a {\em Rayleigh quotient}.  Rayleigh quotients are used to characterize the extreme eigenvalues $\lam_1$ and $\lam_n$ of  $M$: $\lam_1=\min_{\bx\ne 0}\frac{\bx^TM\bx}{\bx^T\bx}$ and $\lam_n=\max_{\bx\ne 0}\frac{\bx^TM\bx}{\bx^T\bx}$.  Rayleigh quotients are used in \cite{AH16, AH18} to show that if $\dQspec(G)=\dQspec(G')$ and $G$ is transmission regular, then so is $G'$ (see Section \ref{s:cospec}).
 \ms

{\bf Interlacing} Recall that the eigenvalues of a symmetric matrix $M\in\Rnn$ are denoted by $\lam_1(M)\le\dots\le\lam_n(M)$.  Let $M(i)$ be the $(n-1)\x(n-1)$ matrix obtained from $M$ be deleting row and column $i$ from $M$.  Then it is well known that the eigenvalues of $M(i)$ interlace those of $M$ \cite[Theorem 4.3.17]{HJ}:\vspace{-5pt}
\[\lam_1(M)\le\lam_1(M(i))\le\lam_2(M)\le\lam_2(M(i))\le \dots\lam_{n-1}(M)\le\lam_{n-1}(M(i))\le\lam_n(M).\vspace{-5pt}\]
Interlacing was used in \cite{GRWC15} to show that a hypercube with a leaf added has at most five distance eigenvalues.  It can be problematic to use interlacing for any of the other matrices because the result of deleting a vertex affects the entire matrix for $\DQ, \DL$, and $\nDL$ (but not for $\D$).\ms

{\bf Conjugation and inertia}
Symmetric matrices $M,M'\in\Rnn$ are {\em conjugate} if there exists an invertible matrix $C$ such that $M=C^TM'C$.  For any graph $G$, it is immediate from the definition of $\nDL(G)$ that $\nDL(G)$ and $\DL(G)$ are conjugate. It is well known that conjugate matrices have the same inertia \cite[Theorem 4.5.8 (Sylvester's Law of Inertia)]{HJ}.  Thus $\nDL(G)$ is positive semidefinite because $\DL(G)$  is positive semidefinite.\ms

%%%%%%%%%%%%%%%%%%%%%%%%%%%
\section{Strongly regular graphs, distance regular graphs, and transmission regular graphs}\label{s:SRG} %Leslie
%{\red \cite{AP15, GRWC15}}

Strongly regular graphs and distance regular graphs provide useful examples, including as graphs with few distinct eigenvalues. However, we note examples of graphs that are not distance regular, and thus not strongly regular, yet have few distinct distance eigenvalues are also known  (see, %Section \ref{ss:few}). 
 for example, \cite{GRWC15}). 
There also techniques specific to strongly regular and distance regular graphs for computing the spectrum of the distance matrix, and thus the other three spectra, since all these graphs are transmission regular.  We also present some observations relevant to all transmission regular graphs in Section \ref{ss:tr}.

%------------------------------------------------------------------
\subsection{Strongly regular graphs}\label{ss:SRG}
A graph is {\em regular} (or {\em $k$-regular}) if every vertex has degree $k$. 
 A connected $k$-regular graph $G$ of order $n$ is {\em strongly regular} with parameters
$(n, k, a,c)$ if  $k<n-1$, %\footnote{What we have defined is often called {\em primitive} strongly regular graph.  The difference is whether  the complete graph ($k=n-1$) as strongly regular, but we exclude it since its complement is disconnected and it has only two distinct eigenvalues.} 
every pair of adjacent vertices has $a$ common neighbors,
and every pair of distinct nonadjacent vertices has $c$ common neighbors.  %We assume both $G$ and its complement $\Gc$ are connected, which implies $G\ne K_n$ and $c>0$.  
Throughout this section,  $G$ is a connected {strongly regular} with parameters
$(n, k, a,c)$.
It is well-known that  the three distinct eigenvalues of $G$ are  $\tau=\frac{1}{2}(a-c-\sqrt{(a-c)^2+4(k-c)})$, $\theta=\frac{1}{2}(a-c+\sqrt{(a-c)^2+4(k-c)})$, and $k$, with multiplicities  $m_\tau=\frac{1}{2}
\lp n-1+\frac{2k+(n-1)(a-c)}{\sqrt{(a-c)^2+4(k-c)}}\rp$, $m_\theta=\frac{1}{2}
\lp n-1-\frac{2k+(n-1)(a-c)}{\sqrt{(a-c)^2+4(k-c)}}\rp$, and $1$, respectively.

 Since $G$ is connected, $\diam(G)=2$.   Thus the distance eigenvalues of $G$ in increasing order are
{\scriptsize  
\[   \theta_{\D}=\frac{ -4-a+c-\sqrt{(a-c)^2+4(k-c)}}2, \
\tau_{\D}=\frac{ -4-a+c+\sqrt{(a-c)^2+4(k-c)}}2, \ \rho_{\D}=2n-2-k\vspace{-5pt}\]}
\hspace{-4pt}with multiplicities $m_\theta, m_\tau,$ and 1, respectively (see Remark \ref{r:diam2}).
In this section we 
consistently  list the eigenvalues in increasing order and use notation to associate the transformed versions of eigenvalues $\theta $ and $\tau$ with the originals.

 Since a strongly regular graph is transmission regular (with transmission $t(G)=2n-2-k$), it is easy to determine  the eigenvalues of $\DL(G)$,  $\nDL(G)$, and $\DQ(G)$ from those of $\D(G)$, with the corresponding multiplicities, i.e., $m_\theta$ is also the multiplicity of $\theta_{\DL}$,  $\theta_{\nDL}$, $\theta_{\DQ},$ and $m_\tau$ is also the multiplicity of $\tau_{\DL},$  $\tau_{\nDL}$, and $ \tau_{\DQ}$.  The eigenvalues in increasing order are:
{\scriptsize\[ \theta_{\DQ}=\frac {4 n-2 k-8-a-c-\sqrt{(a-c)^2+4(k-c)}}2,\  \tau_{\DQ}=\frac {4 n-2 k-8-a-c+\sqrt{(a-c)^2+4(k-c)}}2,\ \rho_{\DQ}=4n-4-2k.\vspace{-5pt}\]}
{\scriptsize\[0,\  \tau_{\DL}=\frac{ 4n-2k+a-c-\sqrt{(a-c)^2+4(k-c)}}2,\ \theta_{\DL}=\frac{ 4n-2k+a-c+\sqrt{(a-c)^2+4(k-c)}}2.\vspace{-5pt}\]}
{\scriptsize\[0,\  \tau_{\nDL}=\frac {4n-2k+a-c-\sqrt{(a-c)^2+4(k-c)}}{2(2n-2-k)},\ \theta_{\nDL}=\frac {4n-2k+a-c-\sqrt{(a-c)^2+4(k-c)}}{2(2n-2-k)}.\]}
 Observe that $\rho(\DQ(G))=  2t(G)$,  $\rho(\DL(G))=  \theta_{\DL}$, and $\rho(\nDL(G))=  \theta_{\nDL}$.

%There are many examples of strongly regular graphs for which the distance spectrum has been published {\red (see, for example, )}. 
The Petersen graph is a strongly regular graph with parameters $(10,3,0,1)$. The distance, distance signless Laplacian, and distance Laplacian spectra of the Petersen graph $P$ are given in  \cite{AH13}, and it is immediate that $\ndLspec(P)=\{0,1^{(5)},\frac 6 5^{(5)}\}$.

A {\em conference graph} is a strongly regular graph with parameters $k=\frac{n-1}2, a=\frac{n-5}4$ and $ c=\frac{n-1}4$, which implies $n\equiv 1\!\mod 4$ and $m_\theta=m_\tau=\frac{n-1}2$ (in fact,  a strongly regular graph with $m_\theta=m_\tau$ is a standard definition of a {conference graph}); examples of conference graphs can be found in \cite[Chapter 10]{GR01}.  Azarija  showed in  \cite{Az14} that the distance eigenvalues of a conference graph of order $n$ are $ \theta_{\D}=\frac{-3-\sqrt n}2, \tau_{\D}=\frac{-3+\sqrt n}2$, and $\rho_{\D}=\frac{3(n-1)}2$.  He used conference graphs to answer positively a question of Graham and Lov\'asz, as to the existence of a graph for which the number of positive distance eigenvalues exceeds the number of negative distance eigenvalues (counting multiplicities).  He called such graphs {\em optimistic}.

  \begin{thm}\label{t:opt}{\rm\cite{Az14}} If $G$ is a conference graph of order $n\ge 13$, then $n_+(G)>n_-(G)$.
  \end{thm}

\begin{obs} If $G$ is a conference graph of order $n\equiv 1\!\mod 4$, then:
\[\theta_{\DQ}=\frac{1}{2} \left(3 n-\sqrt{n}-6\right), \tau_{\DQ}=\frac{1}{2} \left(3 n+\sqrt{n}-6\right), \rho_{\DQ}={3(n-1)}.\]
\[ \tau_{\DL}=\frac{1}{2} \left(3 n-\sqrt{n}\right), \theta_{\DL}=\frac{1}{2} \left(3 n+\sqrt{n}\right).\]
\[ \tau_{\nDL}=\frac{3 n-\sqrt{n}}{3 (n-1)}, \theta_{\nDL}=\frac{3 n+\sqrt{n}}{3 (n-1)}.\]
 Note that $\lim_{n\to\infty}\tau_{\nDL}=\lim_{n\to\infty}\theta_{\nDL}=1$.
\end{obs}

Optimistic strongly regular graphs have since been characterized, and several additional families of optimistic graphs are identified in \cite{GRWC15}. See Observation \ref{o:tr-opt} for comments on the other $D^*$ spectra of optimistic transmission regular graphs.

% \begin{thm}\label{t:SRG-opt}{\rm\cite{GRWC15}} Let $G$ be a strongly regular graph  with parameters $(n,k,a,c) $.  The graph $G$ is optimistic if and only if $\tau_{\D}>0$ and $m_\tau\ge m_\theta$, i.e., $G$  is optimistic if and only if  $c-\frac {2k} {n-1}\le a < \frac{-4+c+k} 2.$  \end{thm}

%It is well-known that for a strongly regular graph $G$ with parameters $(n, k, a, c)$,  its complement $\Gc$ is also strongly regular with parameters $(n, \ol k, \ol a, \ol c)$, where $\ol k = n-1-k$, $\ol a=n-2-2k+c$, and $\ol c=n-2k+a$.  Furthermore, a conference graph is self-complementary.  

%\begin{thm}{\rm\cite{GRWC15}} Let $G$ be a strongly regular graph with parameters $(n,k,a,c)$.  Both $G$ and $\Gc$ are optimistic if and only if $G$ is a conference graph and $n\ge 13$. \end{thm}

At the other extreme from optimistic strongly regular graphs are strongly regular graphs with exactly one positive distance eigenvalue. It is observed in \cite{GRWC15} that for a  strongly regular graph $G$ with parameters $(n,k,a,c),$ 
 $G=C_5$ or $\tau\le -2$ and thus $G$ has exactly one positive distance eigenvalue if and only if 
   $G=C_5$ or $\tau=-2$. Examples include the complete multipartite graph $K_{2,2,\dots,2}$ (called a {\em cocktail party graph}) and the line graphs $L(K_m)$ and $L(K_{m,m})$; the distance spectra of these graphs are listed in \cite{GRWC15}.  See Observation \ref{o:1poseval} for the impact on the other $D^*$ spectra of transmission regular graphs with one positive distance eigenvalue.

\subsection{Distance regular graphs}\label{ss:dr}
Distance regular graphs are a generalization of strongly regular graphs. The graph $G$ is  {\em distance regular} if for any choice of vertices $u, v$ with $d(u, v) =k$, the number of vertices $w$ such that $d(u, w) =i$ and $ d(v, w) =j$ is independent of the choice of $u$ and $v$.  Distance regular graphs are transmission regular, so $\dQspec(G)$, $\dLspec(G)$, and $\ndLspec(G)$ can be easily determined from $\dspec(G)$ for a distance regular graph. 

The study of distance spectra of distance  regular graphs having exactly one positive distance eigenvalue was  initiated by Koolen and Shpectorov in \cite{KS1994}.  The distance spectra of additional such graphs were determined in \cite{AP15}, and the determination of distance spectra of all such graphs was completed  in  \cite{GRWC15}. Such graphs   are directly related to a metric hierarchy for finite connected graphs (and more generally, for finite distance spaces), which makes these graphs particularly interesting (see \cite{KS1994} for more information). A complete list of the distance spectra of distance regular graphs having exactly one positive distance eigenvalue can be found in \cite{GRWC15}.  This includes the following infinite graph families: cycles, Hamming graphs (which include the strongly regular graphs $L(K_{m,m})$; see Proposition  \ref{t:Ham} for the  spectra of all four distance matrices of a Hamming graph), cocktail party graphs (which are strongly regular), Johnson graphs (which include the strongly regular graphs $L(K_m)$), Doob graphs, halved cubes, and double odd graphs.

Distance spectra of several additional families of distance regular graphs were determined  by Atik and Parighani  in \cite{AP15}, including Hadamard graphs and Taylor graphs,  and in \cite{GRWC15}, including Kneser graphs.  Atik and Parighani also established a tight upper bound on the number of distinct distance eigenvalues of a distance regular graph.

\begin{thm}{\rm\cite{AP15}} If $G$ is distance regular, then $G$ has at most $\diam(G)+1$ distinct distance eigenvalues.
\end{thm}

\subsection{Transmission regular graphs}\label{ss:tr}

In the next two observations we state the equivalent versions for $\DQ(G)$, $\DL(G)$, and $\nDL(G)$ of the conditions that $G$ is optimistic or has exactly one positive distance eigenvalue. Observe that zero as a sorting point for eigenvalues is mapped to $t(G)$ for $\DQ(G)$ and $\DL(G)$, and to 1 for $\nDL(G)$.
\begin{obs}\label{o:tr-opt} Let $G$ be a transmission regular graph.  Then the following statements are equivalent.
\ben
\item $G$ is optimistic, i.e., $\D(G)$ has more positive than negative eigenvalues.
\item The number of  eigenvalues of $\DQ(G)$ greater than $t(G)$ is greater than the number of  eigenvalues of $\DQ(G)$ less than $t(G)$.
\item The number of  eigenvalues of $\DL(G)$ less than $t(G)$ is greater than the number of  eigenvalues of $\DL(G)$ greater than $t(G)$.
\item The number of  eigenvalues of $\nDL(G)$ less than $1$ is greater than the number of  eigenvalues of $\nDL(G)$ greater than $1$.  
\een
\end{obs}

\begin{obs}\label{o:1poseval} Let $G$ be a transmission regular graph.  Then the following statements are equivalent.
\ben
\item  The spectral radius $\rho_{\D}$ is the only   positive  eigenvalue of $\D(G)$.
\item  $\DQ(G)$ has only one eigenvalue greater than $t(G)$.
\item  $0$ is the only eigenvalue of $\DL(G)$ that is less than $t(G)$.
\item  $0$ is the only eigenvalue of $\nDL(G)$ that is less than $1$.
\een
\end{obs}

As noted earlier, most of the initial examples of distance matrices with few distinct eigenvalues involved distance regular graphs.  This led to a search for examples of graph that are not  distance regular whose distance matrices with few distinct eigenvalues and additional properties. An example of a graph $G$  that is transmission regular but not distance regular  where the number of distinct eigenvalues is less than $\diam(G)+1$ was presented in  \cite{GRWC15}.  Since it is transmission regular,   $ \DL(G), \DQ(G)$, and $\nDL(G)$ all have the same number of distinct eigenvalues as $\D(G)$.

%%%%%%%%%%%%%%%%%%%%%%%%%%%
\section{Spectra of specific families of graphs}\label{s:fam} % Leslie
%{\red Includes distance matrix, distance signless Laplacian, distance Laplacian, normalized distance Laplacian.}

In this section we list the spectra of the distance, distance signless Laplacian, distance Laplacian, and normalized distance Laplacian matrices for some specific families of graphs (in addition to specific strongly regular and distance regular graphs discussed in Section \ref{s:SRG}). Some are transmission regular, so the spectra of the other matrices are easily computed from the spectrum of the distance matrix. %, and from the adjacency spectrum when in addition the diameter is two. 
Other computations utilize twins and quotient matrix techniques.  Even though multisets are unordered, we list the eigenvalues in increasing order except where noted.
%In this section, the eigenvalues  are not necessarily listed in increasing order (in some cases, the order depends on the parameters). %For the next three results, the eigenvalues are written in increasing order (even though the spectrum is unordered).
\begin{prop} Let $n\ge 1$. %See {\rm  \cite{GP71}} for $\D(G)$,  {\rm  \cite{AH13}} for $\DL(G)$ and $\DQ(G)$, and {\rm  \cite{R20}} for $\nDL(G)$.  
\bit
\item{\rm  \cite{GP71}} $\dspec(K_n)=\{(-1)^{(n-1)}, n-1\}$.
\item{\rm  \cite{AH13}} $\dQspec(K_n)=\{(n-2)^{(n-1)}, 2n-2\}$.
\item{\rm  \cite{AH13}} $\dLspec(K_n)=\{0,n^{(n-1)}\}$.
\item{\rm \cite{R20}} $\ndLspec(K_n)=\{0,\lp\frac n {n-1}\rp^{(n-1)}\}$.
\eit
\end{prop}

% Theorems \ref{t:twin} and  \ref{t:quot} can be used to establish the next result (although in some cases the cited work uses other methods).  %We illustrate this by establishing $\ndLspec(K_n-e)$. 

\begin{prop}\label{p:star} % See  {\rm  \cite{SI09}} for $\D(G)$,  {\rm  \cite{AH13}} for $\DL(G)$ and $\DQ(G)$, and {\rm  \cite{R20}} for $\nDL(G)$.  
Let $b\ge a\ge 2$.
\bit
\item  {\rm  \cite{SI09}} $\dspec(K_{a,b})=\{-2^{(n-2)}, a+b- 2  - \sqrt{a^2 - ab + b^2}, a+b- 2  + \sqrt{a^2 - ab + b^2}\}$.
\item{\rm  \cite{AH13}} $\dQspec(K_{a,b})= \left\{\frac{1}{2} \lp 5 a+5 b-8-\sqrt{9 a^2-14 a b+9 b^2}\rp\!,\frac{1}{2} \lp 5 a+5 b-8+\sqrt{9 a^2-14 a b+9 b^2}\rp\!,\right.$ $\left.(2a+b-4)^{(a-1)}, (2b+a-4)^{(b-1)}\right\}$ (eigenvalues are not  in increasing order).
\item{\rm  \cite{AH13}} $\dLspec(K_{a,b})=\{0, n,(2a+b)^{(a-1)}, (2b+a)^{(b-1)}\}$.
\item {\rm  \cite{R20}} $\ndLspec(K_{a,b})=\left\{0,\, \frac{2 \left(a^2+a (b-1)+(b-1) b\right)}{(2 a+b-2) (a+2 b-2)}, \, \lp \frac{2b+a}{2b+a-2}\rp^{(b-1)},\ \lp \frac{2a+b}{2a+b-2}\rp^{(a-1)} \right\}$.
\eit
\end{prop}

\iffalse %^^^^^^^^^^^^^^^^ 
\begin{thm}\label{t:Kn-e} For $n\ge 4$,
\bit
\item{\rm  \cite{AH13}} $\dspec(K_n-e)=\left\{-2,(-1)^{(n-3)},\frac{n-1-\sqrt{n^2-2 n+9}}2,\frac{n-1+\sqrt{n^2-2 n+9}}2\right\}$ %=\{}, n-1\}$.
\item{\rm  \cite{AH13}} $\dLspec(K_n-e) =\{n+2,n^{(n-2),0}\}$.
\item{\rm  \cite{AH13}} $\dQspec(K_n-e)=\{(n-2)^{(n-2)},\frac{3n-2-\sqrt{n^2-4 n+20}}2,\frac{3n-2+\sqrt{n^2-4 n+20}}2 \}$.
\item\label{ndL-Kn-e}{\rm  \cite{R20}} $\ndLspec(K_n-e)=\{\frac {n+2} {n}, \lp\frac n{n-1}\rp^{(n-3)}, 0, \frac{n^2-n+2}{(n-1) n}\}$.
\eit
\end{thm}
\fi %^^^^^^^^^^^^^^^^
Since $C_n$ is transmission regular, the distance signless Laplacian spectrum, the distance Laplacian spectrum, and the normalized distance Laplacian spectrum can be readily determined from the distance spectrum  (in fact,   $\D(C_n)$ is a polynomial in $\A(C_n)$, so its spectrum can be determined from the adjacency spectrum), but the formulas depend on the parity of $n$. The distance spectrum (attributed to \cite{FCH01}), the distance signless Laplacian spectrum, and the distance Laplacian spectrum of the cycle are presented in \cite{AH13}, and the normalized distance Laplacian spectrum of the cycle appears in \cite{R20}.  Here we list the distance spectrum (the transmission is the first eigenvalue listed and the eigenvalues  are not  listed in increasing order).
\[\dspec(C_n)= \begin{cases}
\{p^2, 0^{(p-1)}\}\cup\{-\csc^2\lp\frac{\pi(2j-1)}{2p}\rp:j=1,\dots, p\} & \mbox{for } n=2p\ge 4 \\
\{p^2+p\}\cup\{-\frac 1 4\sec^2\lp\frac{\pi j}{2p+1}\rp,-\frac 1 4\csc^2\lp\frac{\pi(2j-1)}{2(2p+1)}\rp:j=1,\dots, p\} & \mbox{for }n=2+1\ge 5
\end{cases}. \]

\iffalse %^^^^^^^^^^^^^^^^ 
\begin{thm}  For $n=2p\ge 4$:
\bit
\item{\rm  \cite{FCH01}}  $\dspec(C_n)=\{p^2, 0^{(p-1)}\}\cup\{-\csc^2\lp\frac{\pi(2j-1)}{2p}\rp:j=1,\dots, p\}$.
\item {\rm  \cite{AH13}} $\dLspec(C_n)=\{0, (p^2)^{(p-1)}\}\cup\{p^2+\csc^2\lp\frac{\pi(2j-1)}{2p}\rp:j=1,\dots, p\}$.
\item {\rm  \cite{AH13}} $\dQspec(C_n)=\{2p^2, (p^2)^{(p-1)}\}\cup\{p^2-\csc^2\lp\frac{\pi(2j-1)}{2p}\rp:j=1,\dots, p\}$.
\item{\rm \cite{R20}} $\ndLspec(C_n)=\{0, 1^{(p-1)}\}\cup\{1+\frac 1{p^2}\csc^2\lp\frac{\pi(2j-1)}{2p}\rp:j=1,\dots, p\}$.
\eit
For $n=2+1\ge 5$,
\bit
\item{\rm  \cite{FCH01}}  $\dspec(C_n)=\{p^2+p\}\cup\{-\frac 1 4\sec^2\lp\frac{\pi j}{2p+1}\rp,-\frac 1 4\csc^2\lp\frac{\pi(2j-1)}{2(2p+1)}\rp:j=1,\dots, p\}$.
\item {\rm  \cite{AH13}} $\dLspec(C_n)=\{0\}\cup\{p^2+p+\frac 1 4\sec^2\lp\frac{\pi j}{2p+1}\rp,p^2+p+\frac 1 4\csc^2\lp\frac{\pi(2j-1)}{2(2p+1)}\rp:j=1,\dots, p\}$.
\item {\rm  \cite{AH13}} $\dQspec(C_n)=\{2(p^2+p)\}\cup\{p^2+p-\frac 1 4\sec^2\lp\frac{\pi j}{2p+1}\rp,p^2+p-\frac 1 4\csc^2\lp\frac{\pi(2j-1)}{2(2p+1)}\rp:j=1,\dots, p\}$.
\item{\rm \cite{R20}} $\ndLspec(C_n)=\{0\}\cup\{1+\frac 1 {4(p^2+p)}\sec^2\lp\frac{\pi j}{2p+1}\rp,1+\frac 1 {4(p^2+p)}\csc^2\lp\frac{\pi(2j-1)}{2(2p+1)}\rp:j=1,\dots, p\}$.
\eit
\end{thm}
\fi %^^^^^^^^^^^^^^^^ 

For $r\ge 2$ and $d\ge 1$, the Hamming graph $H(d,r)$ has vertex set consisting of all $d$-tuples of elements taken from $\{0,\dots,r-1\}$, with two vertices adjacent if and only if they differ in exactly one coordinate. Note that  
$H(d,r)$ is isomorphic to $K_r\cp \cdots\cp K_r$ with $d$ copies of $K_r$; $H(d,2)$ is also called the $d$th hypercube and denoted by $Q_d$. 
% Theorem \ref{tr-dspec-cprod} was established in  \cite{I2009} and is applied there to show that  the distance spectrum of the Hamming graph $H(d,n)$ is  \beq{\dspec(H(d,r))}=\left\{ (-r^{d-1})^{(d(r-1))}, 0^{(r^{d}-d(r-1)-1)}, dr^{d-1}(r-1)\right\}.\label{HamDspec}\eeq   
The spectra of the distance signless Laplacian, distance Laplacian, and normalized distance Laplacian matrices can be computed from equation \eqref{eq:tr-equiv} and the distance spectrum is established in  \cite{I2009} (see Section \ref{s:prod}).

\begin{prop}\label{t:Ham} For $d\ge 1,r\ge 2$,
\bit
\item{\rm  \cite{I2009}} $\dspec(H(d,r))=\left\{ (-r^{d-1})^{(d(r-1))}, 0^{(r^{d}-d(r-1)-1)}, dr^{d-1}(r-1)\right\}$.
\item $\dQspec(H(d,r))=\left\{ (dr^{d-1}(r-1)-r^{d-1})^{(d(r-1))}, (dr^{d-1}(r-1))^{(r^{d}-d(r-1)-1)}, 2dr^{d-1}(r-1)\right\} $.
\item $\dLspec(H(d,r)) =\left\{0,(dr^{d-1}(r-1))^{(r^{d}-d(r-1)-1)}, (dr^{d-1}(r-1)+r^{d-1})^{(d(r-1))}\right\}$.
\item $\ndLspec(H(d,r))=\left\{0,1^{(r^{d}-d(r-1)-1)}, \lp\frac{dr^{d-1}(r-1)+r^{d-1}}{dr^{d-1}(r-1)}\rp^{(d(r-1))}\right\}$.
\eit
\end{prop}

The spectra of the %distance, distance signless Laplacian, distance Laplacian, and normalized distance Laplacian 
matrices of $K_n-e$ (the graph obtained from $K_n$ by deleting an edge) can be found in  \cite{AH13} (distance, distance Laplacian and distance signless Laplacian), and \cite{R20} (normalized distance Laplacian). 

The spectra of the distance, distance signless Laplacian, and distance Laplacian matrices of $S_n^+=K_{1,n-1}+e$ (the graph obtained from $K_{1,n-1}$ by adding an edge) can be found in  \cite{AH13}; we illustrate the use of twins and quotient matrices to establish the normalized distance Laplacian spectrum of $S_n^+$.

\begin{prop}\label{p:Sn+e} For $n\ge 5$, let  $ \mu_2, \mu_3$ be the roots of \[p_B(x)=x^2 + \lp\frac{-8 n^2+20 n-7}{2 (n-2) (2 n-3)}\rp x + \lp \frac{2 n^2-n}{(n-1) (2 n-3)}\rp .\]  Then
$\ndLspec(S_n^+)=\lsb \frac {2n-3} {2n-4}, \lp\frac {2n-1}{2n-3}\rp^{(n-4)}\rsb\cup\{0,\mu_2,\mu_3\}$  (eigenvalues are not  in increasing order).
\end{prop}
\bpf  Without loss of generality, let $\deg(v_1)=n-1$ and let $e=\{v_2,v_3\}$. Then $v_2$ and $v_3$ are adjacent twins with transmission $t(v_2)=2n-4$ and $v_4,\dots,v_n$  are independent twins with transmission $t(v_4)=2n-3$.   Thus $\frac {2n-3} {2n-4}\in\ndLspec(S_n^+) $ and $\lp\frac {2n-1}{2n-3}\rp^{(n-4)}\in\ndLspec(S_n^+) $ by Theorem \ref{t:twin}. 
The quotient matrix of $\nDL(S_n^+)$ is \[B=\mtx{ 1 & -\frac{2}{\sqrt{(n-1) (2 n-4)}} & \frac{3-n}{\sqrt{(n-1) (2 n-3)}} \\
 -\frac{1}{\sqrt{(n-1) (2 n-4)}} & 1-\frac{1}{2 n-4} & -\frac{2 (n-3)}{\sqrt{(2 n-4) (2 n-3)}} \\
 -\frac{1}{\sqrt{(n-1) (2 n-3)}} & -\frac{4}{\sqrt{(2 n-4) (2 n-3)}} & 1-\frac{2 (n-4)}{2 n-3}}\!.\]
and the characteristic polynomial of $B$ is 
\[p_B(X)=x^3 + \lp\frac{-8 n^2+20 n-7}{2 (n-2) (2 n-3)}\rp x^2 + \lp \frac{2 n^2-n}{(n-1) (2 n-3)}\rp x,\]
so $\spec(B)=\{0,\mu_2,\mu_3\}$.
Thus $\ndLspec(S_n^+)=\lsb \frac {2n-3} {2n-4}, \lp\frac {2n-1}{2n-3}\rp^{(n-4)}\rsb\cup\{\mu_1=0,\mu_2,\mu_3\}$ by Theorem \ref{p:twin-quot}. 
\epf

%\begin{rem}\label{r:twin-quot} It can be challenging to compute the eigenvalues of the quotient matrix for $\nDL$ in Proposition \ref{p:Sn+e}.  Recall that $\nDL(G)$ is similar to $T(G)^{-1}\DL(G)$ \cite{R20}, and $B$ is similar to the quotient matrix of $T(G)^{-1}\DL(G)$.  This was the strategy used to find the characteristic polynomial in {p:Sn+e}\end{rem}

%%%%%%%%%%%%%%%%%%%%%%%%%%%
\section{Spectral radii}\label{s:rho} %Carolyn
% {\red Includes distance matrix, distance Laplacian, distance signless Laplacian, normalized distance Laplacian.\\
% There is a lot of literature on this but I know very little.  However, Derek did a lit search for his GRWC 2017 problem.}
% \cite{H17}

The distance,  distance signless Laplacian, and distance Laplacian matrices  have some similar properties for the spectral radius that the normalized distance Laplacian does not share. For example, the spectral radius of $\D(G)$,  $\DQ(G)$, and $\DL(G)$ is edge addition monotonically decreasing, whereas $\rho\left(\nDL(G)\right)$ is not edge addition monotonically increasing (see the graph $KPK(n_1,n_2,n_3)$ defined below).

\begin{thm}\label{EdgeRemoval}
Let $G$ be a connected graph on $n$ vertices with $u,v\in V(\dig)$ and $uv\not\in E(\dig)$.  Then
\begin{itemize}
    \item {\rm\cite{PR90}} $\dev_i(G)\geq \dev_i(G+(u,v))$ for all $1\leq i\leq n$. % and specifically, $\rho\left(\D(G)\right)\geq \rho\left(\D(G+uv)\right)$,
    \item {\rm\cite{AH13}} $\dqev_i(G')\geq \dqev_i(G+(u,v))$ for all $1\leq i\leq n$. % and specifically, $\rho\left(\DQ(G)\right)\geq \rho\left(\DQ(G+uv)\right)$.
    \item {\rm\cite{AH13}} $\dlev_i(G')\geq \dlev_i(G)+(u,v)$ for all $1\leq i\leq n$. % and specifically, $\rho\left(\DL(G)\right)\geq \rho\left(\DL(G+uv)\right)$.
\end{itemize} 
Furthermore, $\rho\left(\D(G)\right)> \rho\left(\D(G+uv)\right)$ and $\rho\left(\DQ(G)\right)> \rho\left(\DQ(G+uv)\right)$. \end{thm}
\bpf Observe that the cited statements imply $\rho\left(\D^*(G)\right)\ge \rho\left(\D^*(G+uv)\right)$ for $\D^*$ one of $\D,\DQ$, or $\DL$.  Here we prove the last statement. % using the method used to establish analogous results for the spectral radius of  the distance and  distance signless Laplacian matrices of digraphs (see Proposition \ref{p:edge-mon-dig}).  
Let $\D^*$ be one of $\D$ or $\DQ$.  Then  $\D^*(G+uv)\ge 0$ and has a positive eigenvector $\bx$ for $\rho(\D^*)$.  Since $\D^*(G)=\D^*(G+uv)+M$ with $M\gneq 0$, 
\[\rho(\D^*(G))\ge  \frac{\bx^T(\D^*(G+uv)+M)\bx}{\bx^T\bx}=\frac{\bx^T\D^*(G+uv)\bx}{\bx^T\bx}+\frac{\bx^TM\bx}{\bx^T\bx}>\frac{\bx^T\D^*(G+uv)\bx}{\bx^T\bx}.\qedhere\] \epf

Although $\D(G)$ and  $\DQ(G)$ are edge addition monotonically {\em strictly} decreasing, $\DL(G)$ is not, as seen in the next example.

\begin{ex}\label{DL-edgeadd-notstrict} Recall that $S_n^+$ is obtained by adding an edge $e$ to the star $K_{1,n-1}$. As shown in \cite{AH13}, 
$\dLspec(K_{1,n-1})=\{0, n,(2n-1)^{(n-2)}\}$ and $\dLspec(S_n^+)=\{0, n,(2n-3),(2n-1)^{(n-3)}\}$.  Thus for $n\ge 4$, $\rho\left(\DL(K_{1,n-1})\right)= \rho\left(\DL(K_{1,n-1}+e)\right)$.
\end{ex}

Much of the study of spectral radii for the matrices $\D(G),\DQ(G), \DL(G)$, and $\nDL(G)$ has been focused on finding extremal values among connected graphs on $n$ vertices and families of graphs which achieve these values. 
Theorem \ref{EdgeRemoval} has as an immediate consequence that the graph with maximum spectral radius for $\D(G)$, $\DQ(G)$, and $\DL(G)$ must be a tree. In fact, it is known that this maximum is achieved uniquely by the path graph $P_n$.

\begin{thm}
For all connected graphs $G$ on $n$ vertices, the graph $P_n$ is the unique graph which maximizes
\[  \rho\left(\D(G)\right) {\rm\cite{PR90}}; \  \rho\left(\DQ(G)\right) \mbox{and }  \rho\left(\DL(G)\right) {\rm\cite{DN15}}.\]
\end{thm}

% The graphs with largest $\nDL$-spectral radius for $n\leq 10$ are found in \cite{R20}.  
 Define $KPK_{n_1,n_2,n_3}$ for $n_1,n_3\geq 1$, $n_2\geq 2$ to be the graph formed by taking the vertex sum of a vertex in $K_{n_1}$ with one end of the path $P_{n_2}$ and the vertex sum of a vertex in $K_{n_3}$ with the other end of $P_{n_2}$ (see Figure \ref{Fig:KPK}). Note the number of vertices is $n=n_1+n_2+n_3-2$ and $KPK_{1,n,1}=KPK_{2,n-1,1}=KPK_{2,n-2,2}=P_{n}$. It is shown in \cite{R20} that graphs with largest $\nDL$-spectral radius for $n\leq 10$ are of the form $KPK_{n_1,n_2,n_3}$,  and the graph maximizing  $\rho\left(\nDL(G)\right)$ is not a tree for $n\geq 6$. %The graph that achieves maximum spectral radius for $\nDL(G)$
 This shows that   $\rho\left(\nDL(G)\right)$ is not edge addition monotonically decreasing.  The graph that achieves maximum spectral radius for $\nDL(G)$ is not known, but it was conjectured in \cite{R20}. 

\begin{figure}
\begin{center}
\includegraphics[scale=.6]{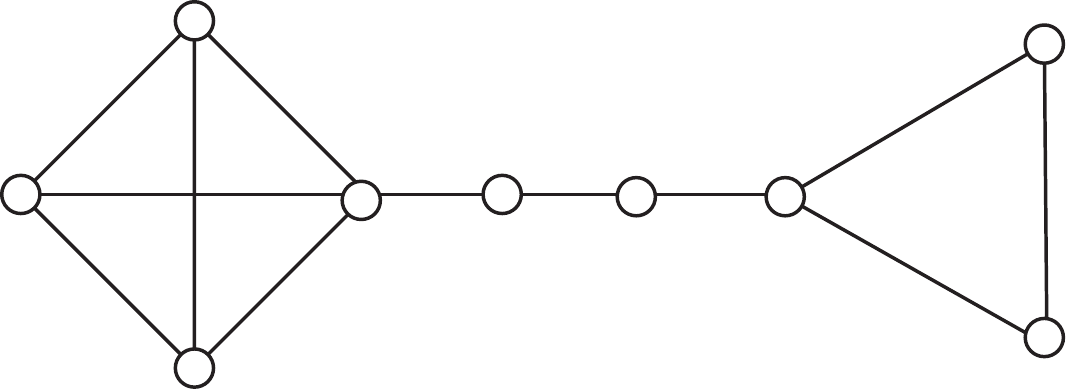}\vspace{-4pt}
\end{center}
\caption{$KPK_{4,4,3}$\label{Fig:KPK}}\vspace{-8pt}
\end{figure}

\begin{conj}{\rm\cite{R20}}
The maximum $\nDL$ spectral radius achieved by a graph on $n$ vertices tends to $2$ as $n\to\infty$ and is achieved by $KPK_{n_1,n_2,n_3}$ for some $n_1+n_2+n_3=n+2$.%\vspace{-8pt}
\end{conj}

For the distance,  distance signless Laplacian, and distance Laplacian matrices, the minimum spectral radius value is known to be  achieved only by the complete graph $K_n$.  This is immediate for $\D(G)$ and $\DQ(G)$ from Theorem \ref{EdgeRemoval} as is the fact that $K_n$ achieves the minimum $\rho\left(\DL(G)\right)$.

\begin{thm}
Let $G$ be a connected graph on $n\geq 2$ vertices. Then,
\begin{itemize}
    \item {\rm\cite{PR90}} $\rho\left(\D(G)\right)\geq \rho\left(\D(K_n)\right)= n-1$,
    \item {\rm\cite{AH16}} $\rho\left(\DQ(G)\right)\geq \rho\left(\DQ(K_n)\right)= 2n-2$,
    \item {\rm\cite{AH13}} $\rho\left(\DL(G)\right)\geq \rho\left(\DL(K_n)\right)= n$,
\end{itemize}
and in each case, equality holds if and only if $G=K_n$. 
\end{thm}

For the normalized Laplacian, the complete graph still achieves the minimum spectral radius. Uniqueness has not been shown; however, it is known that any graph achieving minimum normalized distance Laplacian spectral radius would be $\nDL$-cospectral to $K_n$.

\begin{thm}{\rm\cite{R20}}\label{minRadNormal}
Let $G$ be a connected graph on $n\geq 2$ vertices. Then,\vspace{-5pt}
\[\rho\left(\nDL(G)\right)\geq \rho\left(\nDL(K_n)\right)= \frac{n}{n-1}.\vspace{-5pt}\] Furthermore, if $\rho\left(\nDL(G)\right)=\frac{n}{n-1}$, then  $\spec_{\nDL}(G)=\left\{0,\frac{n}{n-1}^{(n-1)}\right\}$.
\end{thm}

%The fact that the $K_n$ is the unique graph achieving minimum spectral radius for $\D(G), \DQ(G), \DL(G),$ and $\nDL(G)$, together with Theorem \ref{minRadNormal}, leads to the following conjecture.

\begin{conj}{\rm\cite{R20}}\label{conj:nDL-Kn}
Let $G$ be a connected graph on $n\geq 2$ vertices. Then, $\rho\left(\nDL(G)\right)=\frac{n}{n-1}$ if and only if $G$ if the complete graph $K_n$.
\end{conj}

Theorem \ref{minRadNormal} shows that the spectral radius of $\D(G)$,  $\DQ(G)$, and $\DL(G)$ grows with $n$. Unlike the other three matrices, the normalized distance Laplacian has a fixed upper bound on its spectral radius (independent of $n$). 

\begin{thm}{\rm\cite{R20}}\label{t:nDL-rho-ub}
For all connected graphs $G$ on $n$ vertices, $\rho\left(\nDL(G)\right)\leq 2$ and for $n\geq 3$, $\rho\left(\nDL(G)\right)< 2$.
\end{thm}

Recall that the $i$th row sum of $\D(G)$ is the transmission of the $i$th vertex, $r_i(\DQ(G))=2t(v_i)=$ the $i$th absolute row sum of $\DL(G)$.
By Perron-Frobenius theory $t_{\min}\le \rho(\D(G))\le t_{\max}$ and $2t_{\min}\le \rho(\DQ(G))\le 2t_{\max}$, where $t_{\min}$ and $t_{\max}$ are  the minimum and  maximum transmission among vertices of $G$. The spectral radius of each of the four matrices is  bounded by the maximum absolute row sum, so  $\rho\lp\DL(G)\rp\le 2t_{\max}$.
%\begin{obs} For a connected graph $G$, let  $t_{\max}$ be the maximum transmission.\[  \rho\lp\D(G)\rp \le t_{\max},\  \rho\lp\DQ(G)\rp\le 2t_{\max}, \mbox{  and }  \rho\lp\DL(G)\rp\le 2t_{\max}.\]\end{obs}
Rayleigh quotients  (see Section \ref{ss:PF}) can be applied to the distance and distance signless Laplacian matrices to obtain a bound that is tight only when the graph $G$ is transmission regular. We prove the result for $\D$ here; the proof is analogous to the proof for $\DQ$ in \cite{AH16}. 
\begin{thm}\label{thm:RadTransReg}
For a connected graph $G$, let $t_{\min}$ be the minimum transmission, $t_{\max}$ be the maximum transmission, and $\overline{t}$ be the average transmission. 
\begin{itemize}
\item $t_{\min}\leq \overline{t}\leq \rho\left(\D(G)\right) \leq t_{\max}$ and $\rho\lp\D(G)\rp=\ol t$  if and only if $G$ is transmission regular.
\item {\rm\cite{AH16}} $2t_{min}\leq 2\overline{t}\leq \rho\left(\DQ(G)\right) \leq 2t_{\max}$ and $\rho\lp\DQ(G)\rp=2\ol t$  if and only if $G$ is transmission regular.
\end{itemize}
%Furthermore 
\end{thm}

\begin{proof}
 Observe that $t_{\max}$ is the maximum row sum of $\D(G)$. Applying the row sum bound and the Rayleigh quotient with the vector $\bone$ of all 1s, we have
\[t_{\max}\ge \rho\left(\D(G)\right)=\max_{\bx\ne 0}\frac{\bx^T\D(G)\bx}{\bx^T\bx}\geq\frac{\bone^T\D(G) \bone}{\bone^T\bone}=\frac{1}{n}\sum_{i=1}^{n}t(v_i)=\ol t\ge t_{\min}.\]
It $\rho\lp\D(G)\rp=\ol t$, then $\bone$ is an eigenvector for $\rho(\D(G))$ and $G$ is transmission regular. If $G$ is transmission regular, then $\bone$ is an eigenvector for the eigenvalue $t(G)={\ol t}=\rho\left(\D(G)\right)$.
\end{proof}

A related bound has been established for digraphs (see Theorem \ref{t:dig-rho-bds}), and applies to graphs because a graph can be viewed as a doubly directed digraph (see Section \ref{s:dig}).

\begin{thm}
Let $G$ be a strongly connected digraph with vertices $v_1,\dots,v_n$ and transmissions $t(v_1)\leq \dots\leq t(v_n)$. Then
\begin{itemize}
    \item {\rm\cite{LS13}} $\sqrt{t(v_1)t(v_2)}\leq \rho(\D(G))\leq \sqrt{t(v_{n-1})t(v_n)}$ and
    \item {\rm\cite{LMW17}} $t(v_1)+t(v_2)\leq \rho(\DQ(G))\leq t(v_{n-1})+t(v_{n})$
\end{itemize}
and in each case, one of the inequalities holds if and only if $G$ is transmission regular.
\end{thm}

%%%%%%%%%%%%%%%%%%%%%%%%%%%
\section{Cospectrality}\label{s:cospec} % Carolyn
%{\red Can we add distance signless Laplacian?}

For a matrix $\D^*$, two graphs $G$ and $H$ are {\em $\D^*$-cospectral} if $\spec\left(\D^*(G)\right)=\spec\left(\D^*(H)\right)$. If $G$ and $H$ are $\D^*$-cospectral, they are called {\em $\D^*$-cospectral mates} (or just cospectral mates if the choice of $\D^*$ is clear). The number of connected graphs with such a mate has been computed for 10 and fewer vertices for the distance, distance signless Laplacian, distance  Laplacian, and normalized distance Laplacian matrices; see Table \ref{tab:CospecNumber}.%\vspace{-8pt}

\begin{table}[h!]
    \centering
    \caption{Number of connected graphs with a $\D^*$-cospectral mate with respect to each matrix. Counts for $\D,\DQ,\DL$ from \cite{AH18}, counts for $\nDL$ from \cite{R20}.   \label{tab:CospecNumber}}\vspace{10pt}
    
\begin{tabular}{|c|c|c|c|c|c|c|c|c|c|c}
\hline
    &$\#$ connected&  && &    \\
    $n$ &graphs& $\D$ & $\DQ$& $\DL$ & $\nDL$   \\
    \hline
    3&2&0&0&0&0\\
    4&6&0&0&0&0\\
    5&21&0&2&0&0\\
    6&112&0&6&0&0\\
    7&853&22&38&43&0\\
    8&11,117&658&453&745&2\\
    9&261,080&25,058&8,168&19,778&8\\
    10&11,716,571&1,389,984&319,324&787,851&7538\\
\hline\end{tabular}
 
\end{table}

A graph $G$ is {\em determined by its $\D^*$ spectrum} if it has no $\D^*$-cospectral mate. Several such graphs have been found for the distance matrix, distance signless Laplacian, and distance Laplacian. No graphs are currently known to be determined by their $\nDL$ spectrum (although Conjecture \ref{conj:nDL-Kn} is equivalent to conjecturing $K_n$ is determined by its $\nDL$ spectrum).
\begin{thm}
The following graphs are determined by their $\D$ spectrum:
\[   K_n,  P_n\ {\rm\cite{PR90}}; K_{n_1,n_2,\dots,n_k} {\rm\cite{JZ14}};  C_n \mbox{ for $n$ odd}, \ol{P_n}, \ol{C_n}\ {\rm\cite{DL18}}; Q_d \ {\rm\cite{HIK16}}. \]
%\begin{itemize}
  %  \item {\rm\cite{PR90}} The complete graph $K_n$.
%    \item {\rm\cite{PR90}} The path $P_n$.
 %   \item {\rm\cite{JZ14}} The complete k-partite graph $K_{n_1,n_2,\dots,n_k}$.
%    \item {\rm\cite{DL18}} The cycle $C_n$ for odd $n$>
%    \item {\rm\cite{DL18}} The complement of $P_n$ and the complement of $C_n$.
%    \item {\rm\cite{HIK16}} The hypercube $Q_d$.
\end{thm}
\begin{thm}
The following graphs are determined by their $\DQ$ spectrum:
\[   K_n\ {\rm \cite{AH16}}; P_n \ {\rm\cite{DN15}};  C_n, K_n-e, Co_{n,3}\ {\rm\cite{AH18}} \]
where $Co_{n,3}$ is   $P_{n-1}$ with an additional leaf appended to one of the penultimate vertices (called a comet).
%\begin{itemize}
%    \item {\rm \cite{AH16}} The complete graph $K_n$.
 %   \item {\rm\cite{DN15}} The path $P_n$.
 %   \item {\rm\cite{AH18}} The cycle $C_n$.
%   \item {\rm\cite{AH18}} The graph $K_n-e$ obtained from $K_n$ by the deletion of an edge $e$.
  %  \item {\rm\cite{AH18}} The comet $Co_{n,3}$, i.e., a path $P_{n-1}$ with an additional leaf appended to one of the penultimate vertices.\end{itemize}
\end{thm}
\begin{thm}
The following graphs are determined by their $\DL$ spectrum:
\[   K_n\ {\rm \cite{AH13}}; P_n \ {\rm\cite{DN15}};  K_{n_1,n_2,\dots,n_k}, K_n-e, Co_{n,3} \ {\rm\cite{AH18}}.  \]
\end{thm}

Since $P_n$ is determined by its $\D^*$ spectrum for $\D^*=\D, \DQ$ and $\DL$, we ask the following question.

\begin{quest} Is $P_n$ determined by its $\nDL$ spectrum?  It is for $n\le 10$ \cite{SageR21}.
\end{quest}

A graph parameter is {\em preserved by $\D^*$-cospectrality} if two graphs that are $\D^*$-cospectral must share the same value for that parameter (it can be numeric or true/false). A great many parameters have been shown to be preserved or not preserved by $\D^*$-cospectrality. It is obvious that two graphs $G$ and $H$ must have the same order to be $\D^*$-cospectral for any matrix $\D^*$. %; therefore graph order is preserved by $\D^*$-cospectrality for $M=\D,\DQ,\DL,\nDL$. 
Similarly, the trace of a matrix $\D^*$, $\tr(\D^*)$, must be preserved by $\D^*$-cospectrality since it is equal to the sum of the eigenvalues.
Some known results are summarized in Table \ref{tab:PreservedCospec}; in this table, a question mark indicates that it has been verified that no example of non-preservation exists on ten or fewer vertices. This verification was performed using {\em Sage} \cite{SageR21}. Next, we  list a source or example for each definitive  answer.

\begin{table}[ht!]
    \centering
   \caption{Some parameters that are known to preserved or not preserved   by $\D^*$-cospectrality. \label{tab:PreservedCospec}}\vspace{5pt}
  
\begin{tabular}{|r|c|c|c|c|c|c|c|c|c|c}
    \hline
   Property  & $\D$ & $\DQ$& $\DL$ & $\nDL$   \\
    \hline
 $\#$ Edges &No&?&No&No\\ % \hline
    Diameter &No&?&No&? \\  %\hline
    Girth &No&?&No&No\\  %\hline
    Planarity &No&No&No&No\\  %\hline
     Wiener index &No&Yes&Yes&No\\  %\hline
     Degree sequence &No&No&No&No\\ % \hline
     Transmission sequence &No&No&No&No\\  %\hline
    Transmission regularity& ? & Yes & No & ?\\ % \hline
    $\#$ connected components in $\overline{G}$ & No & No & Yes & No\\
        \hline
\end{tabular}\vspace{-5pt}
 \end{table}

The number of edges in a graph is not preserved by cospectrality for $\D$ \cite{H17}, $\DL$ \cite{GRWC18}, or $\nDL$ \cite{R20} and the diameter of a graph has been shown not to be preserved by  $\D$-cospectrality \cite{GRWC16} and $\DL$-cospectrality \cite{GRWC18}. The girth of a graph is the length of the shortest cycle in the graph. Girth was shown not to be preserved by  $\DL$-cospectrality \cite{GRWC18} and $\nDL$-cospectrality \cite{R20}; we show now in Example \ref{ex:DGirth} that girth is not preserved by $\D$-cospectrality.

\begin{ex}\label{ex:DGirth}
The graphs $G_1$ and $G_2$ in Figure \ref{fig:DGirth} are $\D$-cospectral with distance characteristic polynomial $p_{\D}(x)=x^9 - 112x^7 - 758x^6 - 1994x^5 - 2010x^4 + 184x^3 + 1262x^2 +
193x - 222$. The girth of $G_1$ is 4 and the girth of $G_2$ is 3. \vspace{-8pt}
\end{ex}

\begin{figure}[h!]
    \centering
    \includegraphics[scale=.7]{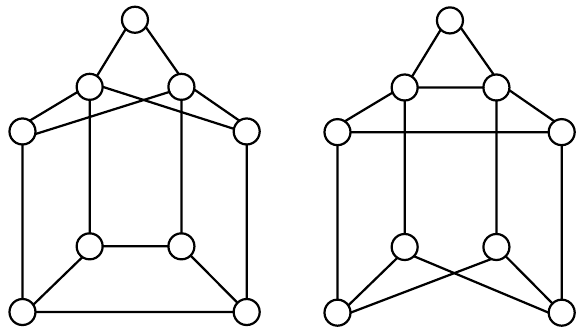}\\
    $G_1$\qquad\qquad \ \ \ $G_2$\vspace{-4pt}
    \caption{A pair of $\D$-cospectral graphs, which show that the girth of $G$ is not preserved by $\D$-cospectrality.}\vspace{-8pt}
    \label{fig:DGirth}
\end{figure}

A graph is planar if it can be drawn in a way such that no edges intersect each other. Planarity was shown not to be preserved by $\DL$-cospectrality in \cite{GRWC18} and $\nDL$-cospectrality in \cite{R20}. We show it is not preserved by $\D$-cospectrality in Example \ref{ex:D} and by $\DQ$-cospectrality in Example \ref{ex:DQPlanar}.

\begin{figure}[h!]
    \centering
    \includegraphics[scale=.7]{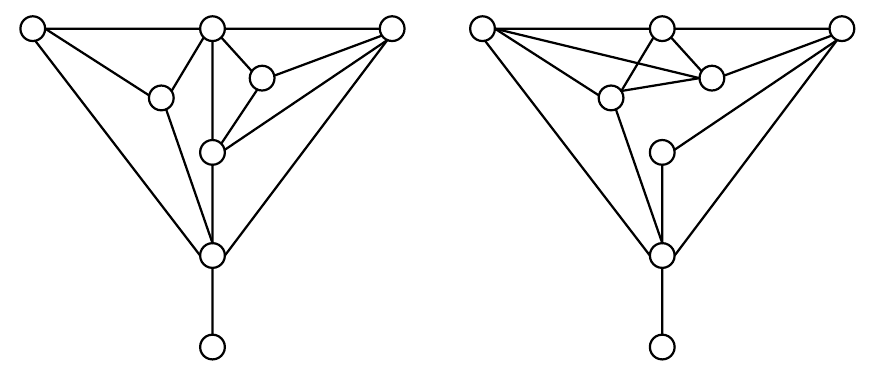}\\
    $G_1$\qquad\qquad\qquad \ \ \ $G_2$\vspace{-4pt}
    \caption{A pair of $\DQ$-cospectral graphs, which show that  planarity is not preserved by $\DQ$-cospectrality.}
    \label{fig:DQPlanar}\vspace{-8pt}
\end{figure}

\begin{ex}\label{ex:DQPlanar}
The graphs $G_1$ and $G_2$ in Figure \ref{fig:DQPlanar} are $\DQ$-cospectral with distance signless Laplacian characteristic polynomial $p_{\DQ}(x)=x^8 - 88x^7 + 3296x^6 - 69002x^5 + 886299x^4 - 7169822x^3 +
35735188x^2 - 100453184x + 122045040$. Observe that $G_1$ is planar and $G_2$ is not planar.
\end{ex}

Recall the Wiener index of a graph is the sum of all pairs of distances in $G$ and $\W(G)=\frac{1}{2}\tr(\DL(G))=\frac{1}{2}\tr(\DQ(G))$. Since trace is preserved by cospectrality for all matrices, this implies the Wiener index is preserved by $\DL$- and $\DQ$-cospectrality. However, it was shown in \cite{GRWC16} that the Wiener index is not preserved by $\D$-cospectrality and it was shown in \cite{R20} that it is not preserved by $\nDL$-cospectrality.  

The degree sequence of a graph is the list of degrees of vertices in the graph in increasing order and the transmission sequence of a graph is the list of transmissions of vertices in the graph in increasing order. The degree sequence and transmission sequence were shown not to be preserved by  $\DL$-cospectrality in \cite{GRWC18} and $\nDL$-cospectrality in \cite{R20}. In Examples \ref{ex:D} and \ref{ex:DQ}, respectively, we show the degree sequence and transmission sequence of a graph are not preserved by $\D$- and $\DQ$-cospectrality. 

In \cite{AH18}, the number of connected components of the graph complement $\overline{G}$ is shown to be preserved by $\DL$-cospectrality in \cite{AH18}. In Examples \ref{ex:D}, \ref{ex:DQ}, and \ref{ex:NDL}, we show the number of connected components of $\overline{G}$ is not preserved by cospectrality for $\D$, $\DQ$, or $\nDL$. 

\begin{ex}\label{ex:D}
The graphs $G_1$ and $G_2$ in Figure \ref{fig:DCompNotPreser} are $\D$-cospectral with distance characteristic polynomial $p_{\D}(x)=x^7 - 39x^5 - 142x^4 - 180x^3 - 72x^2$. Observe that $G_1$ is planar and $G_2$ is not planar. The degree sequences of $G_1$ and $G_2$, respectively, are $[3,4,4,4,5,5,5]$ and $[4,4,4,4,4,4,6]$ and the transmission sequences are $[7, 7, 7, 8, 8, 8, 9]$ and $[6, 8, 8, 8, 8, 8, 8]$. The complement of $G_1$ has one connected component and the complement of $G_2$ has two connected components.%\vspace{-4pt}
\end{ex}

\begin{figure}[h!]
    \centering
    \includegraphics[scale=.45]{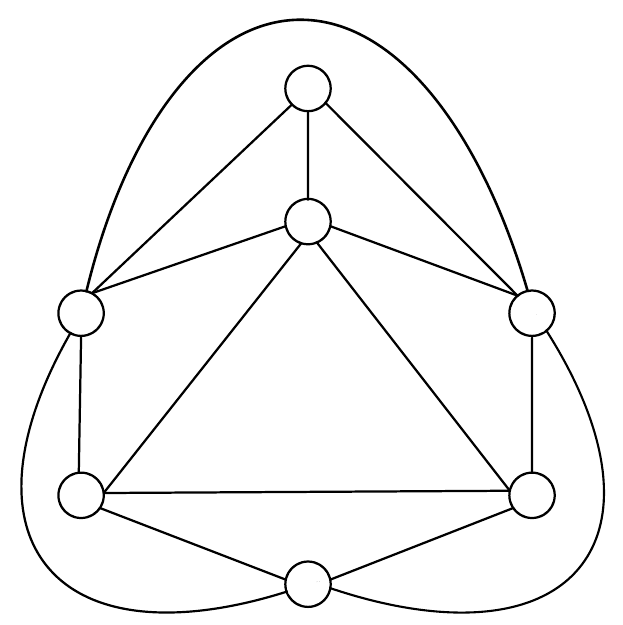}\includegraphics[scale=.45]{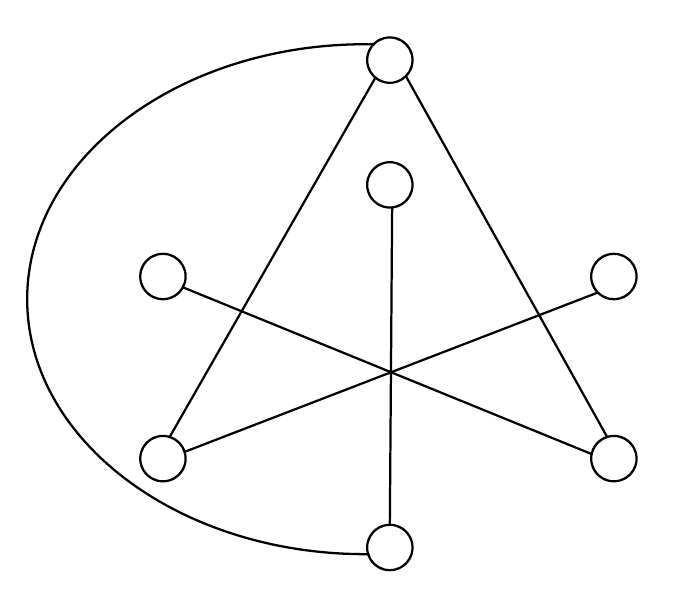}\qquad\includegraphics[scale=.45]{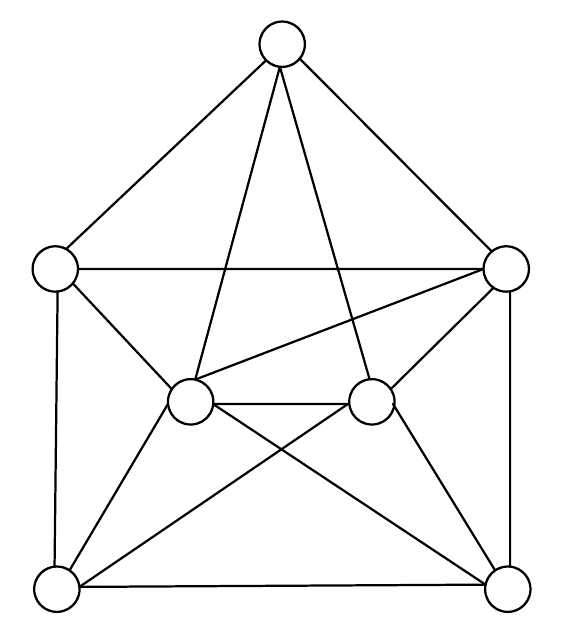}\includegraphics[scale=.45]{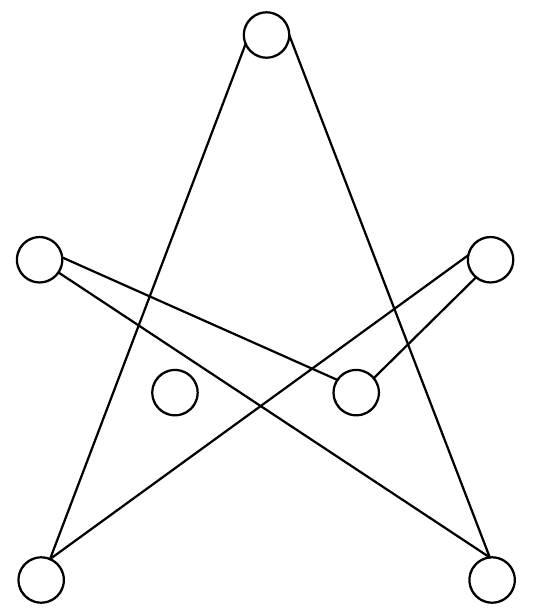}\\
    $G_1$\qquad\qquad\qquad \ \ \ \ ${\ol G_1}$\qquad\qquad\qquad\qquad $G_2$\qquad\qquad\qquad${\ol G_2}$\vspace{-4pt}
    \caption{A pair of $\D$-cospectral graphs $G_1$ and $G_2$ and their complements ${\ol G_1}$ and ${\ol G_2}$, which show that planarity, the degree sequence,  the transmission sequence, and the number of connected components of $\overline{G}$ are not preserved by $\D$-cospectrality.}
    \label{fig:DCompNotPreser}\vspace{-8pt}
\end{figure}

\begin{ex}\label{ex:DQ}
The graphs $G_1$ and $G_2$ in Figure \ref{fig:DQCompNotPreser} are $\DQ$-cospectral with distance signless Laplacian characteristic polynomial $p_{\DQ}(x)=x^5 - 26x^4 + 249x^3 - 1132x^2 + 2480x - 2112$.  The degree sequences of $G_1$ and $G_2$, respectively, are $[1,3,3,3,4]$ and $[2,2,2,4,4]$ and the transmission sequences are $[4, 5, 5, 5, 7]$ and $[4, 4, 6, 6, 6]$. The complement of $G_1$ has two connected components and the complement of $G_2$ has three connected components.
\end{ex}

\begin{figure}[h!]
    \centering
    \includegraphics[scale=.4]{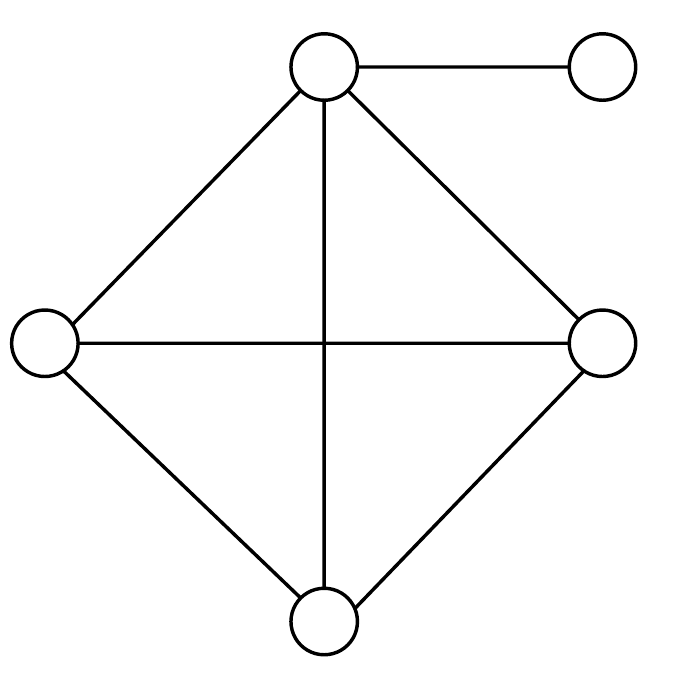}\ \includegraphics[scale=.4]{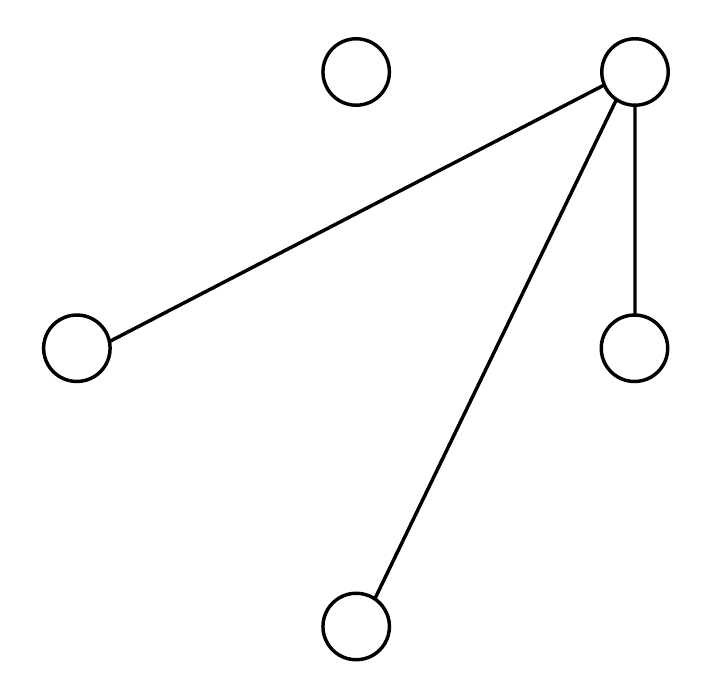}\qquad\includegraphics[scale=.4]{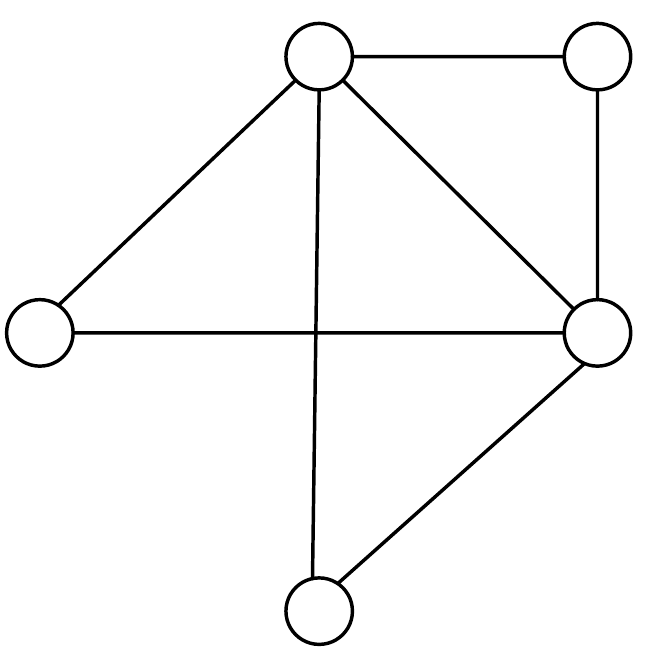}\includegraphics[scale=.4]{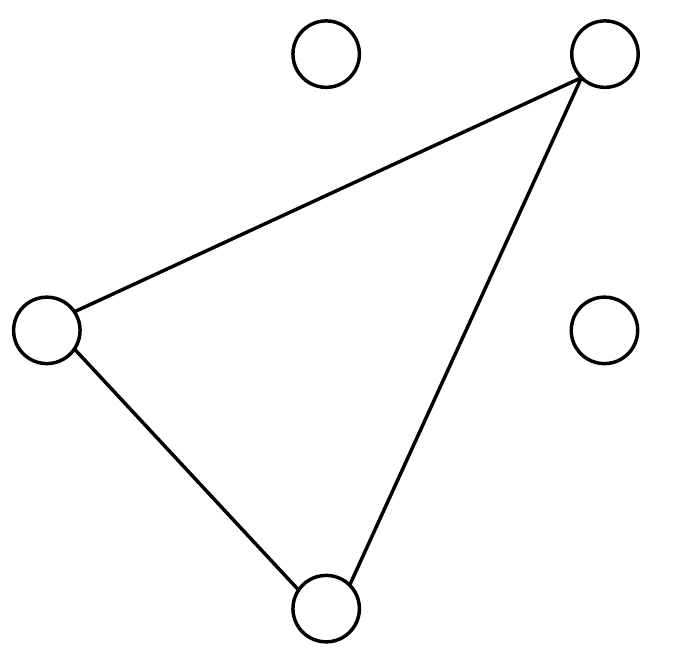}\\
    $G_1$\qquad\qquad\qquad${\ol G_1}$\qquad\qquad\qquad\qquad$G_2$\qquad\qquad\qquad${\ol G_2}$
    \caption{A pair of $\DQ$-cospectral graphs $G_1$ and $G_2$ and their complements ${\ol G_1}$ and ${\ol G_2}$, which show that the degree sequence,  the transmission sequence, and the number of connected components of $\overline{G}$ are not preserved by $\DQ$-cospectrality.}%\vspace{-8pt}
    \label{fig:DQCompNotPreser}
\end{figure}

%\newpage

\begin{ex}\label{ex:NDL}
The graphs $G_1$ and $G_2$ in Figure \ref{fig:NDLCompNotPreser} are $\nDL$-cospectral with normalized distance Laplacian characteristic polynomial $p_{\nDL}(x)=x^{10} - 10x^9 + 222/5x^8 - 2872/25x^7 + 23861/125x^6 -
660126/3125x^5 + 486504/3125x^4 - 230256/3125x^3 + 63504/3125x^2 -
7776/3125x$. The complement of $G_1$ has one connected components and the complement of $G_2$ has five connected components.
\end{ex}

\begin{figure}[h!]
    \centering
    \includegraphics[scale=.55]{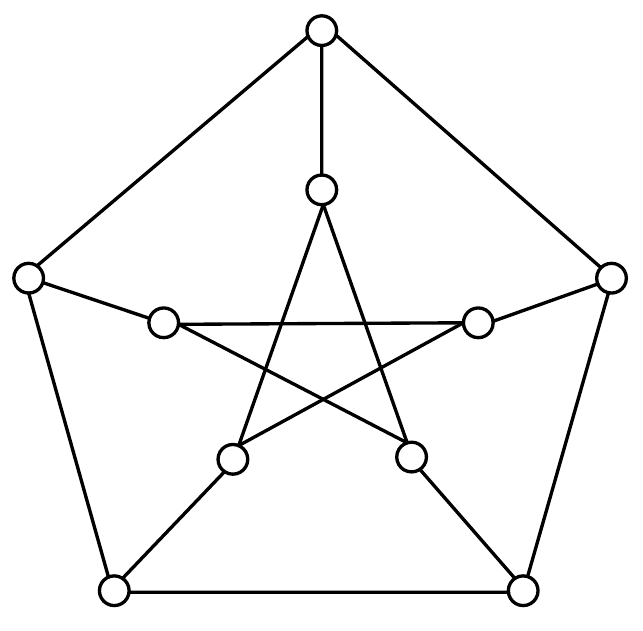}\  \includegraphics[scale=.55]{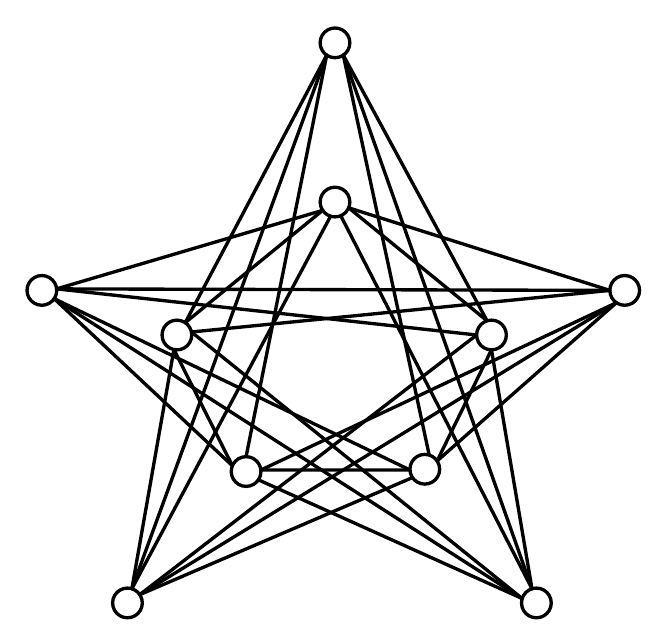}\
    \includegraphics[scale=.25]{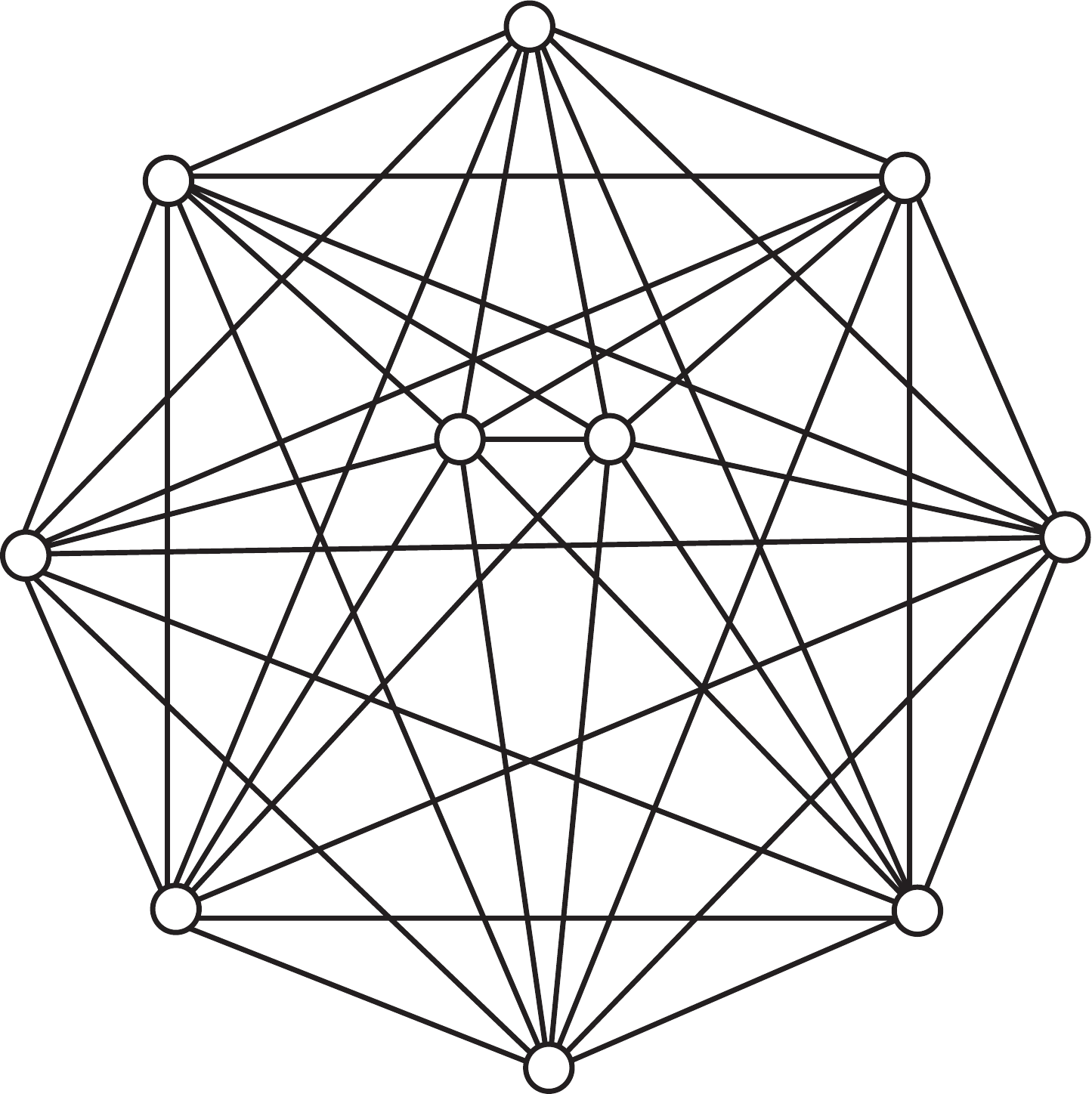}    \includegraphics[scale=.3]{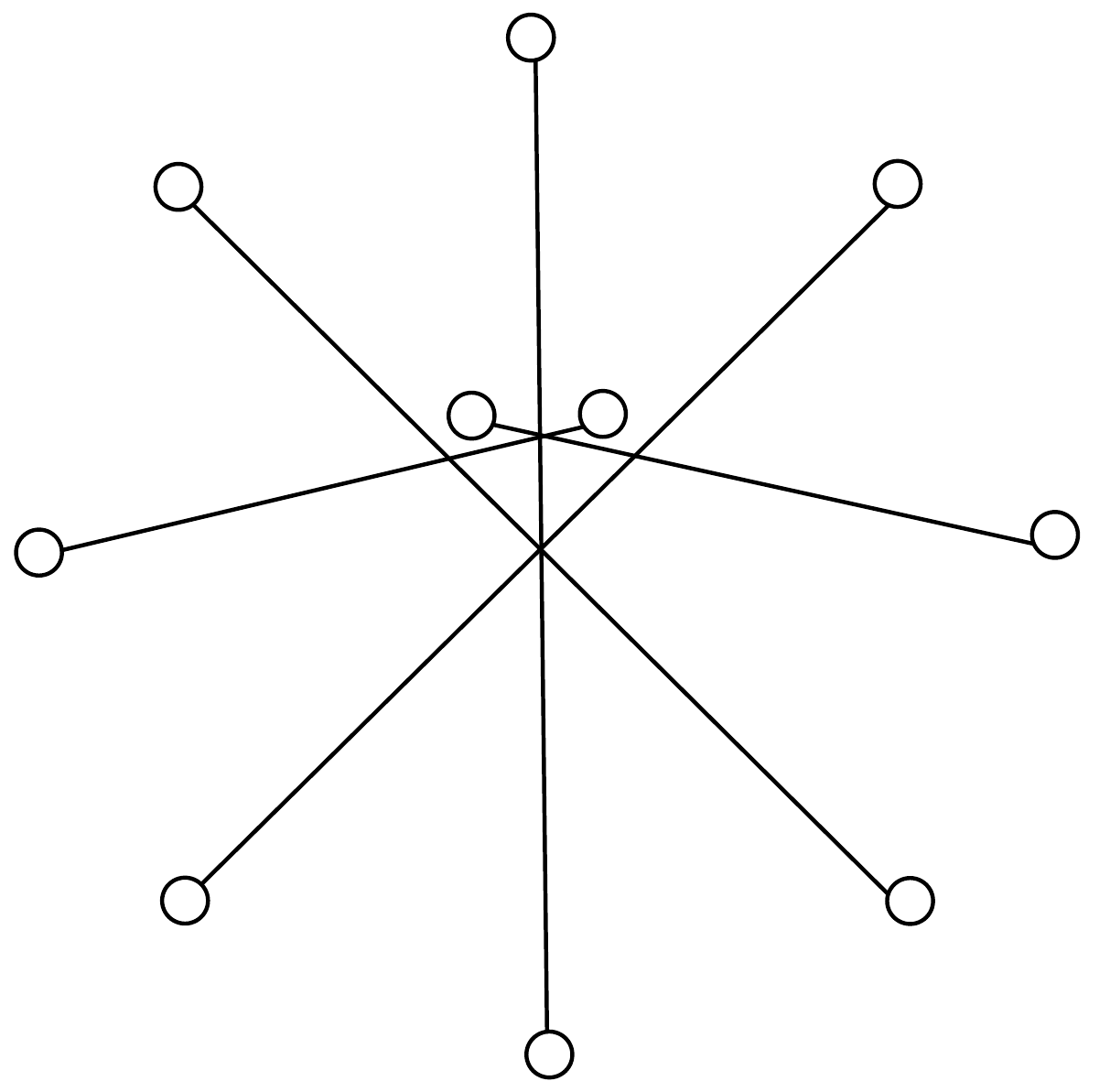}\\
       $G_1$\qquad\qquad\qquad\qquad\qquad${\ol G_1}$\qquad\qquad\qquad\qquad\ \ $G_2$\qquad\qquad\qquad\qquad\qquad${\ol G_2}$
    \caption{A pair of $\nDL$-cospectral graphs  $G_1$ and $G_2$, the Petersen graph and cocktail party graph, respectively, and their complements ${\ol G_1}$ and ${\ol G_2}$, which show that the number of connected components of $\overline{G}$ is not preserved by $\nDL$-cospectrality.}
    \label{fig:NDLCompNotPreser}\vspace{-8pt}
\end{figure}

In \cite{AH18}, it is shown that the property of being transmission regular is preserved by $\DQ$-cospectrality and in \cite{W21}, it is shown that it is not preserved by $\DL$-cospectrality. The proof of preservation by $\DQ$ utilizes Theorem \ref{thm:RadTransReg} and the fact that the trace of $\DQ(G)$ is equal to the sum of the transmissions. A weaker result holds for the distance matrix. In \cite{GRWC16}, it was shown that if two $k$-transmission regular graphs are $\D$-cospectral, then they have the same Wiener index. We observe that this can be improved, since it is immediate that $t(G)=\rho(\D(G))$ for any transmission regular graph $G$.

\begin{rem}
Let $G_1$ and $G_2$ be transmission regular with transmissions $t_1$ and $t_2$ respectively. If $G_1$ and $G_2$ are $\D$-cospectral, then $t_1=\rho\left(\D(G_1)\right)=\rho\left(\D(G_2)\right)=t_2$ and thus $W(G_1)=\frac n 2 t_1=\frac n 2 t_2=W(G_2)$.
\end{rem}

%\begin{proof}Applying Theorem \ref{thm:RadTransReg}, we see $\rho\left(\D(G_1)\right)=t_1$ and $\rho\left(\D(G_2)\right)=t_2$. By $\D$-cospectrality, $\rho\left(\D(G_1)\right)=\rho\left(\D(G_2)\right)$. Finally, $W(G_1)=\frac{n}{2} t_1=\frac{n}{2} t_2=W(G_2)$.\end{proof}

A {\em $\D^*$-cospectral construction} is a process by which $\D^*$-cospectral graphs can be produced. The first such construction was produced for the adjacency matrix by Godsil and McKay, and there has  been much interest in producing such constructions for other matrices.
The first distance cospectral construction was produced by McKay in \cite{M77}. He defined a cospectral construction for trees by identifying the root of one of two particular trees with the root of any rooted tree. For transmission regular graphs, Another distance cospectral construction was given in \cite{GRWC16}, using a lexicographic product of $\D$-cospectral graphs with an independent set or a clique (of course, if $G_1$ and $G_2$ are $\D$-cospectral and $G'$ is regular, then $G_1\lexp G'$ and $G_2\lexp G'$ are $\D$-cospectral by Theorem \ref{dspec-lexprod}). The graphs $G_q=G\lexp {\ol K_q}$ and $G_q^+=G\lexp {K_q}$ are defined in \cite{GRWC16} as obtained from $G$ by replacing vertices in $G$ with cliques and independent sets. In \cite{H17}, Heysse provides two constructions for distance cospectral graphs. One construction uses one of two special graphs and identifies one of their vertices with a vertex in some other graph. The other construction relies on the switching of subgraphs within a larger graph.

In \cite{GRWC18}, a cospectral construction for the distance Laplacian was defined that relied on special properties of a subset of vertices in a graph called {\em cousins}. Let $G$ be a graph of order at least five with $v_1, v_2,v_3, v_4\in V(G)$. Let $C=\{\{v_1,v_2\},$ $\{v_3,v_4\}\}$ and $U(C)=V(G) \setminus \{v_1, v_2, v_3, v_4\}$. Then
$C$ is a  {\em set of cousins} in $G$ if the following conditions are satisfied:
\begin{enumerate}
\item For all $u \in U(C)$, $d_G(u, v_1)=d_G(u, v_2)$ and $d_G(u, v_3)=d_G(u, v_4)$.
\item $\sum_{u \in U(C)} d_G(u, v_1)=\sum_{u \in U(C)} d_G(u, v_3)$.\vspace{-8pt}
\end{enumerate}

\begin{figure}[h!]
    \centering
    \includegraphics{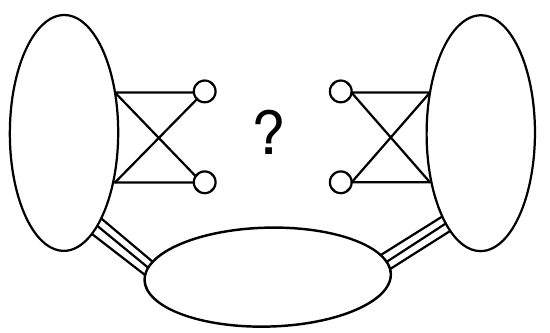}\vspace{-6pt}
    \caption{For vertices that are cousins, the subgraph containing the four vertices could contain any configuration of edges.}
    \label{fig:cousins}\vspace{-4pt}
\end{figure}

Cousins can be considered to be a generalization of twin vertices. Outside of the subset containing vertices $\{v_1, v_2, v_3, v_4\}$, each of the pairs $\{v_1,v_2\}$ and $\{v_3,v_4\}$ is like a pair of twins. However, it is not specified which edges are included in the subgraph on vertices $\{v_1, v_2, v_3, v_4\}$, as visualized in Figure \ref{fig:cousins}. In \cite{GRWC18}, cousins were applied to produce the following $\DL$-cospectral constructions (where $G+e$ denotes $G$ with edge $e$ added).

\begin{thm}{\rm\cite{GRWC18}} \label{thm:CospecCousins}
Let $G$ be a graph with a set of cousins $C=\{\{v_1,v_2\},\{v_3,v_4\}\}$ satisfying the following conditions:
\bit
\item Vertices $v_1, v_2$ are not adjacent and $v_3, v_4$ are not adjacent.
\item The subgraph of $G+v_1v_2$ induced by  $\{v_1,v_2,v_3,v_4\}$ is isomorphic to the subgraph of $G+v_3v_4$  induced by $\{v_1,v_2,v_3,v_4\}$.  
\eit
If $G_1=G+v_1v_2$ and $G_2=G+v_3v_4$  are not isomorphic, then they are  $\DL$-cospectral.
\end{thm}
\begin{thm}{\rm\cite{GRWC18}} \label{thm:CospecCousins2}
Let $G$ be a graph with a set of cousins $C=\{\{v_1, v_2\}, \{v_3, v_4\}\}$ satisfying the following conditions:\vspace{-5pt}
\bit
\item Vertices $v_1, v_3$ are not adjacent and $v_2, v_4$ are not adjacent.
\item The subgraph of $G+v_1v_3$ induced by $\{v_1,v_2,v_3,v_4\}$  is isomorphic to the subgraph of $G+v_2v_4$ induced by $\{v_1,v_2,v_3,v_4\}$  via the permutation $\sigma_1=(14)(23)$.
\item For every  $x \in N(v_1) \cap N(v_2)$ there exists $y\in N(v_3) \cap N(v_4)$ such that $xy \in E(G)$, and for every  $y \in N(v_3) \cap N(v_4)$ there exists $x\in N(v_1) \cap N(v_2)$ such that $xy \in E(G)$.
\eit
% via the permutation $\sigma_1=(14)(23)$ 
If $G_1=G+v_1v_3$ and $G_2=G+v_2v_4$  are not isomorphic, then they are  $\DL$-cospectral.
\end{thm}

It is shown in $\cite{R20}$ that these cousin constructions do not directly extend to the normalized distance Laplacian. However, many $\nDL$-cospectral pairs do contain cousins, so it may be possible to produce a $\nDL$-cospectral construction using cousins given the appropriate additional conditions. In \cite{L20}, Lorenzen extended the application of cousins  by the relaxing the definition. In her paper, using this relaxed definition and given certain special conditions, cospectral constructions are found for the adjacency matrix, combinatorial Laplacian matrix, signless Laplacian matrix, normalized Laplacian matrix, and distance matrix.

%%%%%%%%%%%%%%%%%%%%%%%%%%%
\section{Digraphs}\label{s:dig} %Carolyn

A {\em digraph} $\dig = (V(\dig),E(\dig))$ consists of a  set  $V(\dig) = \{v_1, \dots, v_n\}$ of vertices and a set  $E(\dig)$ of ordered pairs of distinct vertices $(v_i, v_j)$ called {\em arcs} (what we define as  a digraph is sometimes called a {\em simple} digraph because it does not allow {\em loops}, i.e., arcs of the form $(v_i,v_i)$). An arc $(v_i,v_j)$ is called {\em doubly directed} if both $(v_i,v_j)$ and $(v_j,v_i)$ are in $E(\dig)$ and a digraph $\dig$ is {\em doubly directed} if every arc of $\dig$ is   doubly directed.  For a graph $G$,  $\overleftrightarrow{G}$ is the doubly directed digraph obtained from $G$ by replacing every edge $uv$ by the two arcs $(u,v)$ and $(v,u)$;   $\overleftrightarrow{K_n}$ is a {\em complete digraph}.    For $n\ge 1$, a dipath $\overrightarrow{P_n}$ is a digraph with vertex set $V(\overrightarrow{P_n}) = \{v_1, \dots, v_n\}$ and arc set $E(\overrightarrow{P_n}) = \{(v_1,v_2), (v_2,v_3),\dots,(v_{n-1},v_n)\}$. For $n\ge 2$, a dicycle $\overrightarrow{C_n}$ is a digraph with vertex set $V(\overrightarrow{C_n}) = \{v_1, \dots, v_n\}$ and arc set $E(\overrightarrow{C_n}) = \{(v_1,v_2), (v_2,v_3),\dots,(v_{n-1},v_n),(v_n,v_1)\}$.
A digraph is {\em $k$-out-regular} if every vertex $u$ has {\em out-degree} $k$, i.e., there are $k$ arcs of the form $(u,v_j)$;  {\em $k$-in-regular} is defined by replacing the arc $(u,v_j)$ by the arc $(v_j,u)$. A digraph that is both $k$-out-regular and $k$-in-regular is said to be {\em $k$-regular}. 

A digraph $\dig$ is {\em strongly connected} if for every ordered pair of vertices $v_i,v_j$, there is a dipath from $v_i$ to $v_j$  in $\dig$.  In a strongly connected digraph $\dig$ with $v_i,v_j\in V(\dig)$, the \emph{distance} from $v_i$ to $v_j$, denoted by  $d(v_i,v_j)$, is the minimum length (number of arcs) in a dipath starting at  $v_i$ and ending at $v_j$. Unless otherwise stated, all digraphs considered here are assumed to be strongly connected, so that all distances are defined. Observe that it is possible that $d(v_i,v_j)\ne d(v_j,v_i)$ in a digraph. The {\em diameter} of $\dig$, denoted by $\diam(\dig)$, is the maximum distance from one vertex to another in $\dig$ and the {\em girth} $g(\dig)$ is the length (number of arcs) of the shortest dicycle in $\dig$. 

The {\em transmission} of a vertex $u\in V(\dig)$, denoted by $t(u)$, is the sum of the distances from $u$ to all  vertices, i.e. $t(u)=\sum_{v_i\in V(\dig)} d(u,v_i)$. This value is also sometimes called the {\em out-transmission} of $u$; the {\em in-transmission} of  $u$ is $\sum_{v_i\in V(\dig)} d(v_i,u)$. 
A digraph $\dig$ is  {\em transmission regular} or {\em $t$-transmission regular}  if every vertex has transmission $t$. In this case, the common value of the transmission of a vertex is denoted by $t(\dig)$.
A transmission regular digraph is also called {\em out-transmission regular};   $\dig$ is {\em in-transmission regular} or {\em $t$-in-transmission regular}  if every vertex has in-transmission $t$. It is not necessary for a digraph  to be in-transmission regular to be considered transmission regular (this differs from the definition of regular digraphs). There exist digraphs that are out-transmission regular but not in-transmission regular (see Figure \ref{fig:out-not-in} and Example \ref{ex:out-not-in}).  
 \vspace{-5pt}

\begin{figure}[h!]
    \centering
\scalebox{.6}{\includegraphics{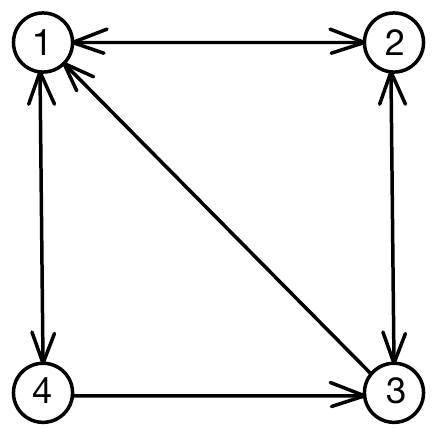}    }\vspace{-5pt}
    \caption{A digraph $\dig_4$ that is out-transmission regular but not in-transmission regular \cite{CCHR20}.
    \label{fig:out-not-in}}\vspace{-5pt}
\end{figure}

The {\em distance matrix} of a strongly connected digraph $\dig$ is  $\D(\dig)=[d(v_i,v_j)]$, i.e.,  the $n \times n$ matrix with $(i,j)$ entry equal to $d(v_i,v_j)$ \cite{GHH77}. %\marginpar{\red need cites}
The distance signless Laplacian, distance Laplacian, and normalized distance Laplacian can be defined analogously to graphs using the distance matrix, i.e., if $T(\dig)$  is the diagonal matrix with $t(v_i)$  as the $i$-th diagonal entry, then   $\DQ(\dig) = T(\dig) + \D(\dig)$ \cite{LMW17}, $\DL(\dig) = T(\dig)  - \D(\dig)$ \cite{CCHR20}, and  $\nDL(G)=\sqrt{T(G)}^{-1}\DL(G)\sqrt{T(G)}^{-1}$.   Equation \eqref{eq:tr-equiv} holds for transmission regular digraphs. Note that unlike for graphs, the distance matrix of a digraph and its variants need not be symmetric, and the eigenvalues can be complex (nonreal); we do not order the eigenvalues of the various distance matrices of a digraph but do use the same notation as for graphs.  

\begin{obs} If $\D^*$ is one of $\D, \DQ, \DL$, or $\nDL$, then $\D^*(G)=\D^*(\overleftrightarrow{G})$.  Thus any graph may be viewed as a doubly directed digraph.

\begin{rem} As is the case for graphs, $0$ is an eigenvalue of $\DL(\dig)$ and $\nDL(\dig)$.  This is immediate for $\DL(\dig)$ because the all ones vector $\bone$ is an eigenvector for $0$.
It is well known that $\spec(AB)=\spec(BA)$ for $A, B\in\Rnn$ \cite[Theorem 2.8]{Zhang}, and thus $\spec(\nDL(\dig))=\spec(T(\dig)^{-1}\DL(\dig))$. Since $\bone$ is an eigenvector of $T(\dig)^{-1}\DL(\dig)$ for $0$, we see that $0$ is an eigenvalue of $\nDL$. %Applying Ger\v sgorin's Disk Theorem to $\spec(T(\dig)\DL(\dig))$ shows $\rho(\nDL(\dig))\le 2$.
\end{rem}

\begin{rem} Recall that for graphs, the eigenvalues are all are nonnegative for the distance signless Laplacian, the distance Laplacian matrix and the normalized distance Laplacian.  By Ger\v sgorin's Disk Theorem, $\re(\dqev_i)\ge 0$ and $\re(\dlev_i)\ge 0$, where $\re(c)$ denotes the real part of a complex number $c$.  Since $\spec(\nDL(\dig))=\spec(T(\dig)\DL(\dig))$, applying Ger\v sgorin's Disk Theorem to $\spec(T(\dig)\DL(\dig))$ shows   that $\re(\ndlev_i)\ge 0$.
\end{rem} 

\end{obs}

Many of the results in the remainder of this section mirror similar results for graphs. However, the techniques used to prove them are frequently different.

\subsection{Techniques}\label{ss:DigTech}

 Since  the real matrices $\D(G)$, $\DL(G)$, $\DQ(G)$, and $\nDL(G)$ are symmetric for graphs, each has real eigenvalues and a basis of orthonormal eigenvectors. These facts are often used in proof techniques. However, the matrices $\D(\dig)$, $\DL(\dig)$, $\DQ(\dig)$, and $\nDL(\dig)$,  are not necessarily symmetric, may have nonreal eigenvalues, and do not always have a basis of eigenvectors.
Instead, proofs often use the JCF. 
The {\em Jordan Canonical Form (JCF)} of a square matrix $A$, denoted by $JCF(A)$, is an upper triangular matrix with non-zero entries only on the diagonal and superdiagonal, made up of Jordan blocks. Each {\em Jordan block} has every entry on the superdiagonal equal to 1 and diagonal entries equal to an eigenvalue $\lambda_i$ of $A$. The number of Jordan blocks corresponding to $\lambda_i$ is  its geometric multiplicity and the sum of the sizes of the Jordan blocks corresponding to $\lambda_i$ is its algebraic multiplicity. For every square matrix $A$, there exists an invertible matrix $C$ such that $C^{-1}AC=JCF(A)$. The Jordan canonical form is a valuable tool for working with the eigenvalues of a digraph since it does not require the algebraic and geometric multiplicity of the eigenvalues to be equal. For more information, see \cite{HJ}. This tool is used extensively in \cite{CCHR20} to prove results on the spectra of digraph products and all of the results in Section \ref{ss:DigProd} employ the JCF of the distance matrix of a digraph to obtain the result.

As described in Section \ref{ss:PF},  Perron-Frobenius theory applies to positive and irreducible nonnegative matrices. Since the distance matrix of a digraph is irreducible and nonnegative and the distance signless Laplacian matrix of a digraph is positive (assuming  order at least two), proofs of results for graphs using  Perron-Frobenius theory can often be adapted to digraphs. For example,  Perron-Frobenius theory shows that $\rho(\D(G)$ and $\rho(\DQ(\dig))$ are simple eigenvalues. The Ger\v sgorin Disk Theorem applies to all complex square matrices, so can be applied to all matrices of digraphs (and is applied, for example, to prove Proposition \ref{t:rho-dig-ub2}).  % (for example, the results in Section \ref{ss:DigRho} 

\subsection{Spectra of products}\label{ss:DigProd}
% \cite{CCHR20}
Many results about products analogous to those for graphs (see Section \ref{s:prod}) hold for digraphs. However, the proofs of the results in this section use the Jordan Canonical Form (see Section \ref{ss:DigTech}). The {\em Cartesian product} of two digraphs $\dig$ and $\dig'$, denoted by $\dig \cp \dig'$, is defined to be the digraph with vertex set $V(\dig) \times V(\dig')$ and arc set \[E(\dig \cp \dig') = \{ ((x,x'),(y,y') )\ | \  x'=y' \mbox{ and } (x,y) \in E(\dig), \mbox{ \bf or }  x=y \mbox{ and  } (x',y')  \in E(\dig')\}. \]

The distance spectrum of $\dig{\cp}\dig'$ for transmission regular graphs is as follows and is analogous to Theorem \ref{tr-dspec-cprod}, the result for graphs.

\begin{thm}\label{thm:TRcartprod-dig_new_new}{\rm\cite{CCHR20}} Let $\dig$ and $\dig'$ be transmission regular digraphs of orders $n$ and $n'$ with transmissions $t$ and $t'$, and let $\spec_{\D}(\dig)=\{t,\dev_2,\dots,\dev_n\}$, $\spec_{\D}(\dig')=\{t',\dev'_2,\dots,\dev'_{n'}\}$.  
Then
\[\spec_{\D}(\dig{\cp}\dig')=\{nt'+n't, n'\dev_2,\dots,n'\dev_n,n\dev'_2,\dots,n\dev'_{n'},0^{(n-1)(n'-1)}\}.\]
\end{thm}

With the hypotheses of Theorem \ref{thm:TRcartprod-dig_new_new}, $\dig\cp\dig'$ is a $(nt'+n't)$-transmission regular digraph and thus  % where $\dig$ and $\dig'$ have transmission $t$ and $t'$ and order $n$ and $n'$ respectively. 
equation \eqref{eq:tr-equiv} can be used  to extend Theorem  \ref{thm:TRcartprod-dig_new_new} to  the distance signless Laplacian, distance Laplacian, and normalized Laplacian matrices.
%Because of this, analogous results to Theorem \ref{thm:TRcartprod-dig_new_new} can easily be obtained for the distance signless Laplacian, distance Laplacian, and normalized Laplacian matrices using equation \eqref{eq:tr-equiv}. 
For transmission regular digraphs for which $\D(\dig)$ and $\D(\dig')$ have a full set of linearly independent eigenvectors, information about the eigenvectors of $\D(\dig{\cp}\dig')$ is also known and can be found in \cite{CCHR20}.

The {\em lexicographic product} of two digraphs $\dig$ and $\dig'$, denoted $\dig \lexp \dig'$, is defined to be the digraph with vertex set $V(\dig) \times V(\dig')$ and arc set \[ E(\dig{{\lexp}}\dig')  = \{ ((x,x'),(y,y') )\ | \  (x,y) \in E(\dig), \mbox{ \bf or }  x=y \mbox{ and  } (x',y')  \in E(\dig')  \}.\]

The distance spectrum of the lexicographic product of two digraphs has been computed for digraphs $\dig$ and $\dig'$ with certain properties (Theorem \ref{DIG-dspec-lexprod} is analogous to Theorem \ref{dspec-lexprod} for graphs). 

\begin{thm}\label{DIG-dspec-lexprod}{\rm \cite{CCHR20}} Let $\dig$ and $\dig'$ be strongly connected digraphs of orders $n$ and $n'$, respectively, such that every vertex of $\dig$ is incident with a doubly directed arc, and $\dig'$ is $k$-out-regular.  Let $\dspec(\dig)=\{ \dev_1,\dots,\dev_{n-1},\dev_n\}$ and $\spec(\A(\dig'))=\{ \alpha'_1,\dots,\alpha'_{n-1},\alpha'_n=k'\}$.  Then 
{\rm \[\dspec(\dig \lexp \dig')=\{ n' \dev_i+2n'-k'-2:i=1,\dots,n\}\cup\{ (-\alpha_j'-2)^{(n))}:j-1,\dots,n'-1\}.\]} \end{thm}

\begin{thm}\label{cor_spectrum_distance_lexp_long_cycle}{\rm\cite{CCHR20}} Let $\dig$ and $\dig'$ be strongly connected digraphs of orders $n$ and $n'$, respectively,  such that  $\dig'$ is $t'$-transmission regular, and $\diam \dig' \leq \ds  g(\dig)$.
Let $\spec_{\D}(\dig)=\{\dev_1,\dev_2,\dots,\dev_n\}$ and $\spec_{\D}(\dig')=\{t',\dev'_2,\dots,\dev'_{n'}\}$.  
Then
{\rm \[\dspec(\dig \lexp \dig')=\left\{n'\dev_i + t', \ i = 1, \dots, n \right\} \cup \left\{{\dev'_j}^{(n)}, \ j = 2, \dots, n' \right\}.\]}
% Given $z\in \mathbb{C}$, define $\tilde z=\frac{z-t'}{n'}$,  %$a=\mult_{\D(\dig)}(\tilde z),\;
% $ g=\gmult_{\D(\dig)}(\tilde z)$, %,\; a'=\mult_{\D(\dig')}(z),\; 
% and $g'=\gmult_{\D(\dig')}(z)$. Then
% \[
% \gmult_{\D(\dig\lexp\dig')}(z)=
% \begin{array}{l}
% \left\{\begin{aligned}
% &g &\mbox{if }\;z\not\in\spec_{\D}(\dig')\setminus\{t'\},\;\tilde z\in\spec_{\D}(\dig);\\
% &ng' &\mbox{if }\; z\in\spec_{\D}(\dig')\setminus\{t'\},\;\tilde z\not\in\spec_{\D}(\dig);\\
% &ng'+g &\mbox{if }\;z\in\spec_{\D}(\dig')\setminus\{t'\},\;\tilde z\in\spec_{\D}(\dig),\;ES_{\D(\dig')}(z)\perp \bone_{n'} ;\\
% &ng' &\mbox{if }\;z\in\spec_{\D}(\dig')\setminus\{t'\},\;\tilde z\in\spec_{\D}(\dig),\;ES_{\D(\dig')}(z)\not\perp \bone_{n'} ;\\
% &0 &\mbox{otherwise}.
% \end{aligned}\right.
% \end{array}
% \]
\end{thm}

  Note that if   $\dig$ and $\dig'$ are transmission regular, or $\dig$ is transmission regular, in addition to the hypotheses in Theorems \ref{DIG-dspec-lexprod} or \ref{cor_spectrum_distance_lexp_long_cycle}, respectively, then $\dig \lexp \dig'$ is transmission regular. Under these conditions, equation \eqref{eq:tr-equiv} can be used  to extend Theorems \ref{DIG-dspec-lexprod} and \ref{cor_spectrum_distance_lexp_long_cycle} to  the distance signless Laplacian, distance Laplacian, and normalized Laplacian matrices. Additionally, the geometric multiplicities of the eigenvalues in Theorem \ref{DIG-dspec-lexprod} and \ref{cor_spectrum_distance_lexp_long_cycle} are known (see \cite{CCHR20}).  If $\D(\dig)$ and $\D(\dig')$ also have a full set of linearly independent eigenvectors, information about the eigenvectors of $\D(\dig \lexp \dig')$ is also known and can be found in \cite{CCHR20}.

% \begin{thm}
% \label{cor_e_vec_distance_lex_bigcycle}{\rm\cite{CCHR20}}
% Let $\dig$ and $\dig'$ be strongly connected digraphs of orders $n$ and $n'$ %, respectively, 
% such that $\dig'$ is $t'$-transmission regular and $\diam \dig' \leq \ds  g(\dig)$. Let $\{\bv_1, \dots, \bv_k\}$ be a linearly independent set of eigenvectors %of $\mathcal{D}(\dig)$ 
% with $\mathcal{D}(\dig) \bv_i = \dev_i \bv_i,$ % \ \dev_i \in \spec_{\mathcal{D}}(\dig)$, 
% and let $\{\bone_{n'},\bv_2', \dots, \bv_{k'}'\}$ be a linearly independent set of eigenvectors %of $\mathcal{D}(\dig')$ 
% with $\mathcal{D}(\dig') \bv_{j}' = \dev_j' \bv_j'.$ % \ \dev_j' \in \spec_{\mathcal{D}}(\dig')$. 
% Then 
% \ben[(1)]
% \item \label{lexevec_1_dist_lex_bigcycle} For $i=1, \dots, k$, \, $\bv_i \otimes \bone_{n'}$ is  an eigenvector of $\mathcal{D}(\dig\lexp\dig')$ corresponding to the eigenvalue  $n'\dev_i + t'$.
% \item \label{lexevec_2_dist_lex_bigcycle} For $j=2, \dots, k'$, \ for $i=1, \dots, k$, define $\gamma_{ij} = \frac{-\dev_i{\bv_j'}^T\bone_{n'}}{t' + n'\dev_i - \dev_j'}$ when $\dev_j' \neq n'\dev_i + t'$.  Then 
% $\bv_i \otimes \bv_j' +\gamma_{ij}\bv_i \otimes \bone_{n'}$   is an eigenvector of $\mathcal{D}(\dig\lexp\dig')$ for the eigenvalue  $\dev_j'$.\vspace{-5pt}\een
% Furthermore, the set of eigenvectors of $\mathcal{D}(\dig\lexp\dig')$ described in \eqref{lexevec_1_dist_lex_bigcycle} and \eqref{lexevec_2_dist_lex_bigcycle} is linearly independent.
% \end{thm}

\subsection{Directed strongly regular graphs and few distinct eigenvalues}\label{ss:digSRG}
% \cite{D88}, \cite{CCHR20}

As is the case with graphs (see Section \ref{ss:SRG}), directed strongly regular graphs are a well structured family of digraphs that are of particular interest since they have only 3 distinct eigenvalues. Duval \cite{D88} defined a {\em directed strongly regular graph}, here denoted by $\dig(n,k,s,a,c)$, to be a digraph $\dig$ of order $n$ such that \vspace{-3pt}
\[\A(\dig)^2=sI_n+a \A(\dig) + c (J_n-I_n-\A(\dig)) \text{  and  } \A(\dig)J_n=J_n\A(\dig)=kJ_n. \vspace{-3pt}\] 

Duval's definition necessitates that directed strongly regular graphs are $k$-regular and and each vertex is incident with $s$ doubly directed arcs. For vertices $v$ and $u$ such that $(v,u)$ is an arc in $E(\dig)$, the number of directed paths of length two from $v$ to $u$ is $a$. For vertices $v$ and $u$ such that $(v,u)$ is not an arc in $E(\dig)$, the number of directed paths of length two from $v$ to $u$ is $c$. Duval originally used the notation $\dig(n,k,\mu,\lambda,t)$, but the notation $\dig(n,k,s,a,c)$ is used in $\cite{CCHR20}$ and will be used here, where $\lam=a$, $\mu=c$, and $t=s$.

Directed strongly regular graphs are transmission regular with transmission $2n-2-k$. Since the spectra of $\DQ$, $\DL$, and $\nDL$ can easily be computed from the spectrum of $\D$ for a transmission regular digraph using equation \eqref{eq:tr-equiv}, the results in this section will focus on the distance matrix.

Duval computed the next formula for the eigenvalues of $\A(\dig(n,k,s,a,c))$. 

\begin{thm}{\rm \cite{D88}}\label{t:DSRG-specA}
%The three   eigenvalues of $\A(\dig)$ for  $\dig=\dig(n,k,s,\lambda,\mu)$ are
Let $\dig=\dig(n,k,s,a,c)$. The spectrum of $\A(\dig)$ consists of the three   eigenvalues 
{\scriptsize \[\tau=\frac{1}{2}\left( a -c - \sqrt{(c-a)^2+4(s-c)} \right),\, \theta=\frac{1}{2}\left(a -c + \sqrt{(c-a)^2+4(s-c)} \right), \text{ and } k,\]} with multiplicities  {\scriptsize \[m_\tau=\frac{1}{2}
\lp n-1+\frac{2k+(n-1)(a-c)}{\sqrt{(a-c)^2+4(s-c)}}\rp,\, m_\theta=\frac{1}{2}
\lp n-1-\frac{2k+(n-1)(a-c)}{\sqrt{(a-c)^2+4(s-c)}}\rp, \text{ and } 1,\]} respectively.
\end{thm}

For $k$-regular digraphs of diameter at most 2, the eigenvalues of $\D(\dig)$ can be written in terms of the eigenvalues of $\A(\dig)$. The following result is analogous to the result for graphs (see Section \ref{s:diam2}) and was used to determine the eigenvalues of $\D(\dig)$ for direct strongly regular graphs.

\begin{prop}\label{Diam2Prop}{\rm\cite{CCHR20}}
Let $\dig$ be a $k$-regular digraph  of order $n$ and diameter at most $2$ with $\spec_\A(\dig)=\{k,\alpha_2,\dots,\alpha_{n}\}$.  Then $\spec_\D(\dig)=\{2n-2-k,-(\alpha_2+2),\dots,-(\alpha_{n}+2)\}$. \end{prop}

Since directed strongly regular graphs are $k$-regular and have diameter at most 2, the eigenvalues of $\D(\dig(n,k,s,a,c))$ can be computed using the eigenvalues of $\A(\dig(n,k,s,a,c))$, as follows.

\begin{prop}\label{cor:SRG_dist_spec}{\rm\cite{CCHR20}}
Let $\dig=\dig(n,k,s,a,c)$. The spectrum of $\D(\dig)$ consists of the three eigenvalues
{\scriptsize \[\tau_{\D}=-2-\frac{1}{2}\left( a -c - \sqrt{(c-a)^2+4(s-c)} \right),\, \theta_{\D}=-2-\frac{1}{2}\left(a -c + \sqrt{(c-a)^2+4(s-c)} \right), \text{ and } \rho(\D(\dig))=2n-2-k,\]} with multiplicities  $m_\tau$, $m_\theta$, and 1, respectively.
\end{prop}

Theorem \ref{thm:TRcartprod-dig_new_new} was applied in \cite{CCHR20} to construct an infinite family of graphs with few eigenvalues. The following result produces a digraph of order $n^{\ell}$ that has exactly 3 eigenvalues. 

\begin{prop}\label{p:cp-DSRG}{\rm\cite{CCHR20}} Suppose $\dig$ is a transmission regular digraph of order $n$ with $\spec_\D(\dig)=\{t=\partial_1, \partial_2^{(m)}, 0^{(n-1-m)}\}$.
 Define $\dig_{\ell}=\dig\cp\dots\cp\dig$, the Cartesian product of $\ell$ copies of $\dig$. Then the order of $\dig_\ell$ is $n^\ell$ and  $\spec_{\D}(\dig_{\ell})=\{\ell t\,n^{\ell-1},\left(\partial_2\,n^{\ell-1}\right)^{(m\ell)},0^{(n^\ell-1-m\ell)}\}$.
\end{prop}

The directed strongly regular graph $\dig=\dig(8,4,3,1,3)$, as shown in Figure \ref{fig:DSRG}, has spectrum $\spec_{\D}(\dig)=\{10,-2^{(5)},0^{(2)}\}$. Applying Proposition \ref{p:cp-DSRG}, $\dig_\ell$ has order $8^\ell$ and   $\spec_{\D}(\dig_{\ell})=\{10\ell(8^{\ell-1}),\left(-2(8^{\ell-1})\right)^{(5\ell)},0^{(8^{\ell}-1-5\ell)}\}$.

\begin{figure}[h!]
    \centering
\scalebox{1}{\includegraphics{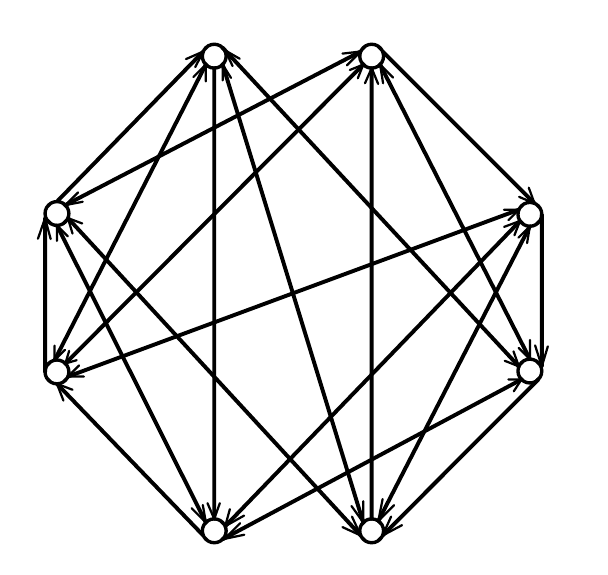}    }
    \caption{$\dig(8,4,3,1,3)$}
    \label{fig:DSRG}
\end{figure}

%\subsection{Spectra of families of digraphs}\label{ss:DigFam} {\red add this?}

\subsection{Spectral radii of digraphs}\label{ss:DigRho}
In the distance  literature, the study of the spectral radius of a digraph has focused on the distance and distance signless Laplacian matrices. Here we summarize these results and provide upper bounds for the spectral radii of the  distance Laplacian and normalized  distance Laplacian matrices.  Recall that since $\D(\dig)$ and $\DQ(\dig)$ are irreducible nonnegative and positive respectively, Perron-Frobenius theory applies. % (see Section \ref{ss:DigTech}).

As is the case with graphs, $\rho\left(\D(\dig)\right)$ and $\rho\left(\DQ(\dig)\right)$ are edge addition monotonically strictly decreasing. 
\begin{prop}\label{p:edge-mon-dig}
Let $\dig$ be a strongly connected digraph with $u,v\in V(\dig)$ and $(u,v)\not\in E(\dig)$. Then
\begin{itemize}
    \item {\rm\cite{LSYZ12}} $\rho\left(\D(\dig)\right)>\rho\left(\D(\dig+(u,v))\right)$ and
    \item {\rm\cite{LMW17}} $\rho\left(\DQ(\dig)\right)>\rho\left(\DQ(\dig+(u,v))\right)$.
\end{itemize}
\end{prop}

Since every graph can be viewed as a doubly directed graph and $\nDL(G)$ is not edge addition monotonically decreasing for graphs, $\nDL(\dig)$ is not edge addition monotonically decreasing for digraphs.
For a graph $G$, $\rho\left(\DL(G)\right)\ge \rho\left(\DL(G+(u,v))\right)$. This raises the analogous question for digraphs.
\begin{quest} Is $\rho\left(\DL(\dig)\right)\ge \rho\left(\DL(\dig+(u,v))\right)$ for every digraph $\dig$?
\end{quest}
The answer to this question is yes  for digraphs of order at most five  \cite{SageR21}. 

Again, analogously to graphs, the minimum value of the spectral radius of $\D(\dig)$ and $\DQ(\dig)$ is achieved uniquely by the complete digraph, and this follows from Proposition \ref{p:edge-mon-dig}.
\begin{thm}\label{thm:DigMin}
For all strongly connected digraphs $\dig$, 
\begin{itemize}
    \item {\rm\cite{LS13}} $\rho(\D(\dig))\geq n-1$
    \item {\rm\cite{LMW17}} $\rho(\DQ(\dig))\geq 2(n-1)$
\end{itemize}
and equality holds if and only if $\dig$ is the complete digraph $\overleftrightarrow{K_n}$.
\end{thm}

While $\rho(\D(G))$ and $\rho(\DQ(G))$ are maximized by the path $P_n$, the dipath $\overrightarrow{P_n}$ is not strongly connected and thus $\D(\overrightarrow{P_n})$ and $\DQ(\overrightarrow{P_n})$ are not defined. Rather, $\rho(\D(\dig))$ and $\rho(\DQ(\dig))$ are maximized by the dicycle $\overrightarrow{C_n}$. %, the digraph with arc set $E=\{v_1v_2,v_2v_3,\dots,v_{n-1}v_n,v_nv_1\}$.

\begin{thm}\label{thm:DigMax}
For all strongly connected digraphs $\dig$, 
\begin{itemize}
    \item {\rm\cite{LS13}} $\rho(\D(\dig))\leq \frac{n(n-1)}{2}$
    \item {\rm\cite{LMW17}} $\rho(\DQ(\dig))\leq n(n-1)$
\end{itemize} 
and equality holds if and only if $\dig$ is the dicycle $\overrightarrow{C_n}$.
\end{thm}

For most digraphs, the bounds in Theorems \ref{thm:DigMin} and \ref{thm:DigMax} are not very tight. By Perron-Frobenius theory $t_{\min}\le \rho(\D(\dig))\le t_{\max}$ and $2t_{\min}\le \rho(\DQ(\dig))\le 2t_{\max}$, where $t_{\min}$ and $t_{\max}$ are  the minimum and  maximum transmission among vertices of $\dig$.  The next result provides slightly tighter bounds in terms of the two smallest and two largest transmissions of vertices in the digraph. 

\begin{thm}\label{t:dig-rho-bds}
Let $\dig$ be a strongly connected digraph with vertices $v_1,\dots,v_n$ and transmissions $t(v_1)\leq \dots\leq t(v_n)$. Then
\begin{itemize}
    \item {\rm\cite{LS13}} $\sqrt{t(v_1)t(v_2)}\leq \rho(\D(\dig))\leq \sqrt{t(v_{n-1})t(v_n)}$ and
    \item {\rm\cite{LMW17}} $t(v_1)+t(v_2)\leq \rho(\DQ(\dig))\leq t(v_{n-1})+t(v_{n})$
\end{itemize}
and in each case, one of the inequalities holds if and only if $\dig$ is transmission regular.
\end{thm}

Finally, we provide bounds on the spectral radii of distance Laplacians and normalized distance Laplacians.

\begin{prop}\label{t:rho-dig-ub2} Let $\dig$ be a strongly connected digraph. % of order $n$.  
Then 
\bit 
\item $\ds \rho(\DL(\dig))\le 2\max_{v\in V(\dig)} t(v)$. % and $\re(\ndlev_i)\ge 0$ for $i=1,\dots,n$.
\item  $\rho(\nDL(\dig))\le 2$. % and $\re(\ndlev_i)\ge 0$ for $i=1,\dots,n$.
\eit
\end{prop}
\bpf The first statement is immediate by Ger\v sgorin's Disk Theorem.
 Recall that $\nDL(\dig)=\sqrt{T(\dig)}^{-1}\DL(\dig)\sqrt{T(\dig)}^{-1}$.  %It is well known that $\spec(AB)=\spec(BA)$ for $A, B\in\Rnn$ \cite[Theorem 2.8]{Zhang}, and thus $\spec(\nDL(\dig))=\spec(T(\dig)\DL(\dig))$. 
 Finally, applying Ger\v sgorin's Disk Theorem to $\spec(T(\dig)\DL(\dig))$ shows $\rho(\nDL(\dig))\le 2$. % and $\re(\ndlev_i)\ge 0$ for $i=1,\dots,n$. \epf
 \epf

\subsection{Cospectral digraphs}\label{ss:DigCospec}

Given a digraph $\dig$, the {\em arc reversal} of $\dig$, denoted by $\dig^T$, is the digraph with  $V(\dig^T)=V(\dig)$ and $E(\dig^T)=\{(v,u): (u,v)\in E(\dig)\}$.  Since $\D(\dig^T)=\D(\dig)^T$, it is immediate that $\dig$ and $\dig^T$ have the same distance spectrum (and thus are $\D$-cospectral if they are not isomorphic).  Note that $\dig$ and $\dig^T$ need not be $\DQ$-cospectral because it is possible that $\DQ(\dig^T)\ne\DQ(\dig)^T$, and similarly for $\DL$ and $\nDL$.  This is illustrated in the next example.

\begin{ex}\label{ex:out-not-in} Let $\dig_4$ be the digraph shown in Figure \ref{fig:out-not-in}.  Then  $\dig_4$ is transmission regular but $\dig_4^T$ is not   \cite{CCHR20}. Various matrices and their spectra are listed in Table \ref{tab:arc-rev}.  Observe that $\dQspec(\dig_4^T)\ne\dQspec(\dig_4)$,$\dLspec(\dig_4^T)\ne\dLspec(\dig_4)$,  and $\ndLspec(\dig_4^T)\ne\ndLspec(\dig_4)$. \vspace{-5pt}

\begin{table}[h!]
    \centering
    \caption{Various distance matrices and spectra for $\dig_4$ and $\dig_4^T$.   \label{tab:arc-rev}}\vspace{5pt}
    
\begin{tabular}{|c|c|c|c|}
\hline
    $\D(\dig_4)$ & $\DQ(\dig_4)$ & $\DL(\dig_4)$ & $\nDL(\dig_4)$   \\
    \hline
       $\mtx{0 & 1 & 2 & 1 \\
 1 & 0 & 1 & 2 \\
 1 & 1 & 0 & 2 \\
 1 & 2 & 1 & 0 }$ & 
 $\mtx{4 & 1 & 2 & 1 \\
 1 & 4 & 1 & 2 \\
 1 & 1 & 4 & 2 \\
 1 & 2 & 1 & 4 }$ & 
 $\mtx{ \phantom{-}4 & -1 & -2 & -1 \\
 -1 & \phantom{-}4 & -1 & -2 \\
 -1 & -1 & \phantom{-}4 & -2 \\
 -1 & -2 & -1 & \phantom{-}4 }$ & 
 $\mtx{\phantom{-}1 & -\frac{1}{4} & -\frac{1}{2} & -\frac{1}{4} \\
 -\frac{1}{4} & \phantom{-}1 & -\frac{1}{4} & -\frac{1}{2} \\
 -\frac{1}{4} & -\frac{1}{4} & \phantom{-}1 & -\frac{1}{2} \\
 -\frac{1}{4} & -\frac{1}{2} & -\frac{1}{4} & \phantom{-}1  }$   \\
 \hline
    $\dspec(\dig_4)$&  $\dQspec(\dig_4)$&  $\dLspec(\dig_4)$&  $\ndLspec(\dig_4)$\\
    \hline
      $\{-2, -1, -1,4\}$& $\{2,3,3,8\}$& $\{0,5,5,6\}$&$\{0,\frac{5}{4},\frac{5}{4},\frac{3}{2}\}$\\
\hline\hline
%%% start transpose
    $\D(\dig_4^T)$ & $\DQ(\dig_4^T)$ & $\DL(\dig_4^T)$ & $\nDL(\dig_4^T)$   \\
    \hline
       $\mtx{0 & 1 & 1 & 1 \\
 1 & 0 & 1 & 2 \\
 2 & 1 & 0 & 1 \\
 1 & 2 & 2 & 0 }$ & 
 $\mtx{3 & 1 & 1 & 1 \\
 1 & 4 & 1 & 2 \\
 2 & 1 & 4 & 1 \\
 1 & 2 & 2 & 5}$ & 
 $\mtx{ \phantom{-}3 & -1 & -1 & -1 \\
 -1 & \phantom{-}4 & -1 & -2 \\
 -2 & -1 & \phantom{-}4 & -1 \\
 -1 & -2 & -2 & \phantom{-}5 }$ & 
 $\mtx{\phantom{-}1 & -\frac{1}{2 \sqrt{3}} & -\frac{1}{2 \sqrt{3}} & -\frac{1}{\sqrt{15}} \\
 -\frac{1}{2 \sqrt{3}} & \phantom{-}1 & -\frac{1}{4} & -\frac{1}{\sqrt{5}} \\
 -\frac{1}{\sqrt{3}} & -\frac{1}{4} & \phantom{-}1 & -\frac{1}{2 \sqrt{5}} \\
 -\frac{1}{\sqrt{15}} & -\frac{1}{\sqrt{5}} & -\frac{1}{\sqrt{5}} & \phantom{-}1 }$   \\
 \hline
    $\dspec(\dig_4^T)$&  $\dQspec(\dig_4^T)$&  $\dLspec(\dig_4^T)$&  $\ndLspec(\dig_4^T)$\\
    \hline
      $\{-2, -1, -1,4\}$& 
      $\{2,\frac{ \left(11-\sqrt{29}\right)}2,3,\frac{ \left(11+\sqrt{29}\right)}2\}$& 
      $\{0,\frac{ \left(11-\sqrt{5}\right)}2,5,\frac{ \left(11+\sqrt{5}\right)}2\}$&
      $\{0,\frac{5}{4},\frac{ \left(165-\sqrt{105}\right)}{120},\frac{ \left(165+\sqrt{105}\right)}{120}\}$\\

\hline\end{tabular}
 
\end{table}
\end{ex}

There is  work being done on distance cospectral digraphs \cite{R21}.  A construction for $\D$-cospectral digraphs is presented there and it is observed for digraphs of small order, a much higher percentage of digraphs have a $\D$-cospectral mate even after excluding pairs of digraphs related by arc reversal.

%\noindent{\red Bibliography copied from various sources as a start}
%%%%%%%%%%%%%%%%%%%%%%%%%%%%%%%%%%%%%%%%%%%

\end{document}